\documentclass{amsart}

\usepackage{enumerate}
\usepackage{microtype}
\usepackage{amssymb,amsmath,array,lineno}

\newtheorem{thm}{Theorem}[section]

\newtheorem{lem}[thm]{Lemma}

\newtheorem{prop}[thm]{Proposition}

\theoremstyle{definition}
\newtheorem{defn}{Definition}

\newtheorem{notn}[thm]{Notation}

\theoremstyle{remark}
\newtheorem{rem}[thm]{Remark}

\DeclareMathOperator{\adj}{adj}
\DeclareMathOperator{\spn}{span}
\newcommand{\rank}{\mathrm{rk}}

\begin{document}

\title{Diophantine approximation of Mahler numbers}
\author{Jason P.~Bell}
\address{Department of Pure Mathematics, University of Waterloo, Waterloo, Canada}
\email{jpbell@uwaterloo.ca}

\author{Yann Bugeaud}
\address{D\'epartement de math\'ematiques, Universit\'e de Strasbourg, Strasbourg, France} 
\email{bugeaud@math.unistra.fr}
%
\author{Michael Coons}%
\address{School of Math.~and Phys.~Sciences, University of Newcastle, Callaghan, Australia}
\email{michael.coons@newcastle.edu.au}
\thanks{Research of J.~P.~Bell was supported by NSERC grant 31-611456, and the research of M.~Coons was supported by ARC grant DE140100223.}

\begin{abstract}
Suppose that $F(x)\in\mathbb{Z}[[x]]$ is a Mahler function and that $1/b$ is in the radius of convergence of $F(x)$ for an integer $b\geq 2$. In this paper, we consider the approximation of $F(1/b)$ by algebraic numbers. In particular, we prove that $F(1/b)$ cannot be a Liouville number. If, in addition, $F(x)$ is regular, we show that $F(1/b)$ is either rational or transcendental, and in the latter case that $F(1/b)$ is an $S$-number or a $T$-number in Mahler's classification of real numbers.
\end{abstract}

\maketitle

\vspace*{6pt}\tableofcontents  

\section{Introduction}

One of the interesting classes of numbers from the perspective of Diophantine approximation is the set of \emph{automatic real numbers}.  In this setting, one has natural numbers $b$ and $k$ with $b,k\ge 2$, and a sequence $f:\mathbb{Z}_{\ge 0}\to \{0,\ldots ,b-1\}$ such that the value $f(n)$ is generated by a finite-state automaton that reads as input the base $k$ expansion of $n$ and outputs a number in the range. We call such a sequence $\{f(n)\}_{n\geq 0}$ {\em $k$-automatic} (or just {\em automatic}) and call the real number whose $n$-th $b$-ary digit is $f(n)$ a {\em $k$-automatic number}. 

In a series of papers in the 1980s, Loxton and van der Poorten \cite{LvdP1982, LvdP1988} used Mahler's method to investigate whether or not an automatic number could be both irrational and algebraic, a problem posed by Cobham \cite{C1968} in 1968. Mahler's method \cite{M1929, M1930a,M1930b,N1996} is a method in transcendence theory whereby one uses a function $F(x)\in\mathbb{Q}[[x]]$ that satisfies a functional equation of the form \begin{equation}\label{MFE} \sum_{i=0}^d a_i(x)F(x^{k^i})=0\end{equation} for some integers $k\geq 2$ and $d\geq 1$ and polynomials $a_0(x),\ldots,a_d(x)\in\mathbb{Z}[x]$ with $a_0(x)a_d(x)\neq 0$, to give results about the nature of the numbers $F(1/b)$ with $b\geq 2$ an integer such that $1/b$ is in the radius of convergence of $F(x)$. We refer to such numbers $F(1/b)$ as {\em Mahler numbers}. Indeed, it is well-known that automatic numbers are special cases of Mahler numbers (see \cite[Theorem 1]{B1994}). 

Some twenty years after Loxton and van der Poorten's work, Adamczewski and Bugeaud~\cite{AB2007, ABL2004} answered the question concerning automatic numbers by substituting Schmidt's Subspace Theorem in place of Mahler's method.  
Adamczewski and Bugeaud showed that such numbers are either rational or transcendental \cite{AB2007}. Their result naturally leads to the question of irrationality measures for automatic numbers. 

Let $\xi$ be a real number. The {\em irrationality exponent}, $\mu(\xi)$, of $\xi$ is defined to be the supremum of the set of real numbers $\mu$ such that the inequality
$$
0<\left|\xi-\frac{p}{q}\right|<\frac{1}{q^\mu}
$$ 
has infinitely many solutions $(p,q)\in\mathbb{Z}\times\mathbb{N}.$ 
Computing the irrationality exponent of a given real number is generally very difficult, although there are several classes of numbers for which irrationality exponents are well-understood.  All rational numbers have irrationality exponent one, and a celebrated theorem of Roth  \cite{R1955} gives that all irrational algebraic numbers have irrationality exponent precisely two.  In fact, 
the set of real numbers with irrationality exponent strictly greater than two has Lebesgue measure zero. Roth's theorem built on work of Liouville \cite{L1844}, who showed that if $\xi$ is an algebraic number of degree $d$ over $\mathbb{Q}$, then $\mu(\xi)\leq d$.  Using this fact, Liouville produced the first examples of transcendental numbers by constructing real numbers with infinite irrationality exponent; numbers with infinite irrationality exponent are now called {\em Liouville numbers} in his honour.

Towards classifying irrationality exponents of automatic numbers as well as settling a conjecture of Shallit \cite{S1999}, Adamczewski and Cassaigne \cite{AC2006} proved that a Liouville number cannot be generated by a finite automaton.

In this paper, we use Mahler's method to provide the following significant generalisation of Adamczewski and Cassaigne's result.

\begin{thm}\label{mainshort} A Mahler number cannot be a Liouville number.
\end{thm}

\noindent Historically, the application of Mahler's method had one major impediment. In order to consider numbers $F(1/b)$ one must ensure that $a_0((1/b)^{k^i})\neq 0$ for all $i\geq 0$; this condition is commonly referred to as {\em Mahler's condition}. Note that in the statement of Theorem \ref{mainshort} no such condition is stated. Indeed, we are able to remove Mahler's condition within the setting of this paper.  As far as we know, this is the first time that Mahler's condition has been removed.  We suspect that our trick, presented in Section 5, has further applications.

Theorem \ref{mainshort} is obtained via a more general quantitative version (see Theorem~\ref{thestuff}) in which we give a bound for the irrationality exponent of the number $F(a/b)$ in terms of information from the functional equation \eqref{MFE} and the rational number $a/b$ when $|a|$ is small enough compared to $b$. Futhermore, the rational approximations constructed to get the quantitative bound are of high enough quality that we may apply a $p$-adic version of Schmidt's Subspace Theorem to extend Adamczewski and Bugeaud's above-mentioned result to a much larger class of numbers. 

To explain this extension formally, let $\mathbf{a}=\{a(n)\}_{n\geq 0}$ be a sequence taking values in a field $\mathbb{K}$. We define the $k$-kernel of $\mathbf{a}$, denoted by $\mathcal{K}_k(\mathbf{a})$, as the set of distinct subsequences of the form $\{a(k^\ell n+r)\}_{n\geq 0}$ with $\ell\geq 0$ and $0\leq r<k^\ell$. Christol \cite{C1979} showed that a sequence is $k$-automatic if and only if its $k$-kernel is finite. We say the sequence $\mathbf{a}$ is {\em $k$-regular} (or just {\em regular}) provided the $\mathbb{K}$-vector space spanned by $\mathcal{K}_k(\mathbf{a})$ is finite-dimensional. The notion of $k$-regularity was introduced by Allouche and Shallit \cite{AS1992, AS2003}. We also use the term {\em regular} to refer to both the function $\sum_{n\geq 0}a(n)x^n$ and its values at rationals $1/b$, with $b\geq 2$ an integer, for $k$-regular sequences $\mathbf{a}$.

\begin{thm}\label{kregdichotomy} An irrational regular number is transcendental.
\end{thm}

As an immediate corollary of this result, we have that if $F(x)$ is a Mahler function satisfying \eqref{MFE} with $a_0(x)$ a nonzero integer, then $F(1/b)$ is either rational or transcendental (see \cite[Theorem~2]{B1994}). 

A particular case of Theorem \ref{kregdichotomy} is the transcendence of irrational automatic numbers, a result first established in \cite{AB2007}. While both proofs ultimately rest on the Schmidt Subspace Theorem, they are quite different; see the discussion after Theorem \ref{thmNotUeffective} for details.

As all regular numbers are Mahler numbers (see Becker \cite[Theorem 1]{B1994}), {Theorem~\ref{mainshort}} shows that regular numbers cannot be too well approximated by rational numbers. In the final result of this paper, we show that one may control the strength of the approximations of regular numbers by algebraic numbers of higher degree. Before stating the result, we need a few definitions.

Mahler \cite{MahCr32} in 1932, and Koksma \cite{Ko39} in 1939, 
introduced two related measures of the 
quality of approximation of a complex 
transcendental number $\xi$ by algebraic numbers.
For any integer $m \ge 1$, we denote by
$w_m (\xi)$ the supremum of the real numbers $w$ for which
$$
0 < |P(\xi)| < \frac{1}{H(P)^{w}}
$$
has infinitely many solutions in integer polynomials $P(x)$ of
degree at most $m$. Here, $H(P)$ stands for the na\"ive height of the
polynomial $P(x)$, that is, the maximum of the absolute values of
its coefficients. Further, we set 
$$w(\xi) = \limsup_{m \to \infty} \frac{w_m(\xi)}{m}.$$ 
According to Mahler \cite{MahCr32}, we say that $\xi$ is an
\vspace{.1cm}
\begin{itemize}
\item $S$-number, if $w(\xi) < \infty$;
\vspace{.1cm}
\item $T$-number, if $w(\xi) = \infty $ and $w_m(\xi) < \infty$ for any
integer $m \ge 1$;
\vspace{.1cm}
\item $U$-number, if $w(\xi) = \infty $ and $w_m(\xi) = \infty$ for some
integer $m\ge 1$.
\end{itemize}
\vspace{.1cm}
Almost all numbers are $S$-numbers in the sense of Lebesgue measure, and 
Liouville numbers are examples of $U$-numbers. The existence of
$T$-numbers remained an open problem for nearly forty years, until it was
confirmed by Schmidt \cite{Schm70,Schm71}.

Following Koksma \cite{Ko39}, for any integer $m \ge 1$, we denote by
$w_m^* (\xi)$ the supremum of the real numbers $w$ for which
$$
0 < |\xi - \alpha| < \frac{1}{H(\alpha)^{w+1}}
$$
has infinitely many solutions in complex algebraic numbers $\alpha$ of
degree at most $n$. Here, $H(\alpha)$ stands for the na\"ive height of $\alpha$,
that is, the na\"ive height of its minimal defining polynomial.
Koksma \cite{Ko39} defined $S^*$-, $T^*$- and $U^*$-numbers as above, using $w_m^*$
rather than $w_m$. He proved that this classification of numbers
is equivalent to Mahler's classification.
For more information on the functions $w_m$ and $w_m^*$, 
the reader is directed to Bugeaud's monograph \cite{BuLiv}. 

We may now state our final result.

\begin{thm}\label{thmNotU}
An irrational regular number is
an $S$-number or a $T$-number.
\end{thm}

Since automatic numbers form a subset of regular numbers, Theorem \ref{thmNotU} implies
that any irrational automatic number is either an $S$-number or a 
$T$-number, a result already established in~\cite{AdBu11}. 

Theorems \ref{mainshort}, \ref{kregdichotomy}, and \ref{thmNotU} combine to give a rather satisfactory answer to the following problem, posed by
Allouche and Shallit \cite[page 195]{AS1992}: 
\vspace{.1cm}
\begin{quote}
{\it Obtain transcendence results for the real numbers $\sum_{n \ge 0} s(n) p^{-n}$,
where $s(n)$ is $p$-regular and $\sum_{n \ge 0} s(n) x^n$ is not a rational function.}
\end{quote}
\vspace{.1cm}

This paper is organised as follows. In Section \ref{UB}, we produce a family of rational function approximations to a power series $F(x)$ satisfying a Mahler functional equation.  We then show that specialisation of our family of rational functions at a given rational number $a/b$ inside the radius of convergence of $F(x)$ produces a family of good approximations of $F(a/b)$ under very general conditions on $a$ and $b$.  In Section \ref{LB}, we show that the rational approximations produced in Section \ref{UB} cannot be too strong and we are able to combine these results in Section \ref{mu} to give an upper bound on the irrationality exponent for $F(a/b)$ in terms of data from the functional equation and the rational number $a/b$ under Mahler's condition. In Section \ref{MahlerCond}, we remove Mahler's condition. In Section \ref{QS}, we prove a quantitative result on the unimodular completion of matrices with polynomial entries which are of use in Section \ref{muauto}. Section \ref{muauto} considers regular (and hence also automatic) numbers; specifically, in the case of a regular function, the results of Section \ref{QS} are used to  provide an alternative method for removing Mahler's condition. This alternative method produces a functional equation based on information from the $k$-kernel, which is used in Section \ref{regAB} alongside an appropriate version of the $p$-adic Schmidt Subspace Theorem to prove general versions of Theorems \ref{kregdichotomy} and \ref{thmNotU}.

\section{Upper bounds on rational approximations}\label{UB}
In this section, we consider power series $F(x)$ satisfying a $k$-Mahler equation and construct a family of rational function approximations of $F(x)$, which we then use to give good rational approximations to $F(\alpha)$ for rational numbers $\alpha$ inside the radius of convergence of $F(x)$.  We approach this problem by looking at finite-dimensional $\mathbb{Q}(x)$-vector subspaces of $\mathbb{Q}((x))$ that are stable under the ring homomorphism $\Delta_k:\mathbb{Q}((x))\to \mathbb{Q}((x))$ given by $x\mapsto x^k$.   Throughout this paper, unless stated otherwise, we take all complex matrix norms $\|\cdot \|$ to be the operator norm; i.e., $\|A\|=\sup_{\|{\bf v}\|=1} \|A{\bf v}\|$, where the norm of a vector ${\bf v}$ is the ordinary Euclidean norm.

\begin{notn} \label{notn: 1} We adopt the following assumptions and notation.
\begin{enumerate}
\item We take $k$ and $d$ to be natural numbers that are both greater than or equal to $2$.
\item We take $a_{i,j}(x) \in\mathbb{Z}[x]$ to be polynomials for $i,j\in\{1,\ldots,d\}$.
\item We take $B(x)$ to be a polynomial with integer coefficients such that $B(0)\neq 0$.
\item We take 
$H:=\max_{1\leq i,j\leq d}\{\deg a_{i,j}(x), \deg B(x)\}$.
\item  We take 
$\mathbf{A}(x)$ to be the $d\times d$ matrix whose $(i,j)$-entry is $a_{i,j}(x)/B(x)$
and we assume that $\det\left(\mathbf{A}(x)\right)$ is a nonzero rational function.
\item We take $1=F_1(x), F_2(x),\ldots,F_d(x)\in \mathbb{Z}[[x]]$ to be power series that are linearly independent over $\mathbb{Q}(x)$ and which satisfy 
 $$\left[F_1(x),\ldots,F_d(x)\right]^T=\mathbf{A}(x)[F_1(x^k), \ldots,F_d(x^k)]^T.$$ 
\item We let ${\bf F}(x)$ denote the $d\times 1$ vector $\left[F_1(x),\ldots,F_d(x)\right]^T$.
\item We take $a,b\in\mathbb{Z}$ with $a\neq 0$, $b\geq 2$, and $\rho:=\log |a|/\log b\in [0,1/(d+1))$ such that $|a/b|$ is in the radius of convergence of $F_i(x)$ for each $i\in\{1,\ldots,d\}$ and $B((a/b)^{k^n})\neq 0$ for all $n\geq 0.$
\item We let $\nu: \mathbb{Q}((x))\to \mathbb{Z}\cup \{\infty\}$ be the valuation defined by $\nu(0)=\infty$ and $$\nu\left(\sum_{n\geq-m} c_nx^n\right):=\inf\{i:c_i\neq 0\}$$ when $\sum_{n\ge -m} c_n x^n\in \mathbb{Q}((x))$ is a nonzero Laurent power series. 
\end{enumerate}
\end{notn}

We begin with a simple norm estimate. 
\begin{lem}\label{Abbound} Assume the notation and assumptions of Notation \ref{notn: 1} and let
$t\in (0,1)$. If $B(t^{k^n})$ is nonzero for every nonnegative integer $n$, then there is a positive constant $C=C(t)>1$ such that for all $n\geq 1$ we have $$\prod_{\ell=0}^{n-1}\|\mathbf{A}(t^{k^{\ell}}) \|< C^n.$$ Moreover, there exists a closed disc containing the origin in which the constant $C$ can be taken to be independent of $t$.
\end{lem}

\begin{proof} Since $B(0)\neq 0$, there is some closed disc containing the origin such that $B(x)\neq 0$ for $x$ in the disc.  Thus by compactness there is some $\varepsilon>0$ such that we have that $\|A(x)\|\le \|A(0)\|+1$ for $x\in [0,\varepsilon]$.

Furthermore, there is some $N_0=N_0(t)$ such that $t^{k^{\ell}}<\varepsilon$ for $\ell\ge N_0$ and in particular, we have $\|\mathbf{A}(t^{k^{\ell}}) \|<\|{\bf A}(0)\|+1$ for all $\ell\ge N_0$.   Since $B(t^{k^n})$ is nonzero for every nonnegative integer $n$, we have that $\|A(t^{k^i})\|$ exists for $i< N_0$. 

Let $$C=C(t):=\max_{0\leq \ell< N_0}\{\|\mathbf{A}(t^{k^{\ell}}) \|, \|{\bf A}(0)\|+1\}.$$ Then we have \begin{equation*}\prod_{\ell=0}^{n-1}\|\mathbf{A}(t^{k^{\ell}}) \|< C^n.\end{equation*} This proves the first assertion.

Notice that for $t\in (0,\varepsilon)$, we can take $N_0=0$ and so we can take $C=\|A(0)\|+1$, which does not depend upon $t$.
\end{proof} 

Our next lemma establishes good rational approximations to power series.

\begin{lem}\label{Fbbound} Assume the notation and assumptions of Notation \ref{notn: 1}.  Then there are $\varepsilon>0$, polynomials $P_{1}(x),\ldots,P_{d}(x),Q(x)\in\mathbb{Z}[x]$ of degree at most $(d-1)(d+2)H$ with $Q(0)=1$, and a positive constant $C=C(\varepsilon)$
such that for $i\in \{1,\ldots ,d\}$ we have
$$\left|F_i(t)-P_i(t)/Q(t)\right|\leq C{t^{d(d+2)H}}$$
for $t\in (0,\varepsilon)$.   
\end{lem}
\begin{proof}  For $i\in\{2,\ldots,d\}$, the theory of simultaneous Pad\'e approximation (see \cite
[Chapter~4]{NS1991} for details) provides polynomials $P_{i}(x)$ and $Q(x)$ of degree each 
bounded by $(d-1)(d+2)H$, and $Q(0)=1$, such that $$\nu\left(Q(x)F_i(x)-P_{i}(x)\right)\geq d(d+2)H.$$ 
Since $F_1(x)=1$, we take $P_1(x)=Q(x)$.  

For $i\in\{1,\ldots,d\}$ we thus have $$\nu\left(F_i(x)-\frac{P_{i}(x)}{Q(x)}\right)\geq d(d+2)H.$$ 

Since $Q(0)=1$ and each of $F_1(x),\ldots ,F_d(x)$ has positive radius of convergence (see \cite[Theorem 31, page 151]{D1993} and \cite[Lemma 4]{BCR2013}), there exists an $R>0$ such that $F_i(x)-{P_{i}(x)}/{Q(x)}$ is analytic on $B(0,R)$ for $i\in \{1,\ldots ,d\}$.
Hence there exist power series $G_1(x),\ldots ,G_d(x)$ that are analytic on $B(0,R)$ such that
$$F_i(x)-\frac{P_{i}(x)}{Q(x)}=x^{d(d+2)H}G_i(x)$$ for $i\in \{1,\ldots ,d\}$.  
Let $\varepsilon\in (0,R)$.  Then there is a positive constant $C$ such that 
$$|G_1(x)|,\ldots ,|G_d(x)| \le C$$ for $|x|\le \varepsilon$.  Thus for $i\in \{1,\ldots ,d\}$,
$$\left|F_i(t)-\frac{P_{i}(t)}{Q(t)}\right|\le C t^{d(d+2)H}$$ whenever $t\in (0,\varepsilon)$.
\end{proof}

\begin{lem}\label{right} Assume the notation and assumptions of Notation \ref{notn: 1} and let $t\in (0,1)$ be less than the radius of convergence of each $F_i(x)$ and satisfy $B(t^{k^n})\neq 0$ for all $n\geq 0$.  Then for each $n\ge 0$ there are polynomials $P_{1,n}(x),\ldots , P_{d,n}(x),Q_n(x)\in\mathbb{Z}[x]$ satisfying:
\begin{enumerate}
\item[(i)] $\max_{1\leq i\leq d}\{\deg P_{i,n}(x),\deg Q_n(x)\}\leq ((d+2)(d-1)+1)Hk^n$;
\item[(ii)] if ${\bf \Phi}_n(x)=[P_{1,n}(x)/Q_n(x),\ldots , P_{d,n}(x)/Q_n(x)]^T$ for $n\ge 0$, then
$${\bf \Phi}_n(x)=\mathbf{A}(x){\bf \Phi}_{n-1}(x^k);$$
\item[(iii)] $Q_n(x)=B(x)B(x^k)\cdots B(x^{k^{n-1}})Q_0(x^{k^n})$;
\item[(iv)] there exists an $\varepsilon>0$ and positive constants $C_0=C_0(\varepsilon)$ and $C_1=C_1(\varepsilon)$, not depending on $t$, such that for $i\in \{1,\ldots ,d\}$ and for all $n$ sufficiently large we have $Q_{n}(t)\neq 0$ and
$$\left|F_i(t)-P_{i,n}(t)/Q_{n}(t) \right|\leq C_1 C_0^n t^{ d(d+2)Hk^n},$$ whenever $t\in(0,\varepsilon)$ and in particular the order of vanishing of $F_i(t)-P_{i,n}(t)/Q(t)$ at $t=0$ is at least ${d(d+2)Hk^n}$; 
\item[(v)] $P_{1,n}(x)=Q_n(x)$.
\end{enumerate}
\end{lem}
\begin{proof} 
By Lemma \ref{Fbbound}, there are $\varepsilon>0$, polynomials $P_{1,0}(x),\ldots,P_{d,0}(x),Q_0(x)\in\mathbb{Z}[x]$ of degree at most $(d+2)(d-1)H$ with $Q_0(0)=1$, and a positive constant $C$
such that for $i\in \{1,\ldots ,d\}$ we have 
$$\left|F_i(t)-P_i(t)/Q_0(t)\right|\leq C{t^{d(d+2)H}}$$ whenever $t\in (0,\varepsilon)$.

We define
\begin{equation}
{\bf \Phi}_0(x) := [P_{1,0}(x)/Q_0(x),\ldots , P_{d,0}(x)/Q_0(x)]^T
\end{equation}
and for $n\ge 1$, we take
\begin{equation}
\label{eq: phin}
\mathbf{\Phi}_{n}(x)=\mathbf{A}(x)\mathbf{\Phi}_{n-1}(x^k).
\end{equation}
We note that there exist integer polynomials $P_{i,n}(x)$ for $i\in \{1,\ldots ,d\}$ and $Q_n(x)$ such that
\begin{enumerate}
\item[(a)] ${\bf \Phi}_{n}(x) = [P_{1,n}(x)/Q_n(x),\ldots , P_{d,n}(x)/Q_n(x)]^T$;
\item[(b)] $Q_n(x)=B(x)Q_{n-1}(x^k)$ for $n\ge 1$.
\end{enumerate}
From (b), we immediately get
$Q_n(x)=B(x)B(x^k)\cdots B(x^{k^{n-1}})Q_0(x^{k^n})$.  Since the entries of ${\bf A}(x)$ are all polynomials of degree at most $H$, we see that if we define
$$d_n:=\max_{1\leq i\leq d}\{\deg P_{i,n}(x),\deg Q_n(x)\},$$ then (\ref{eq: phin}) gives 
$d_{n}\leq kd_{n-1}+H$. By induction we see, using the fact that $d_0\le (d-1)(d+2)H$, that \begin{multline}\label{dnbound}d_n\leq d(d+2)H\cdot k^n+H(k^{n-1}+\cdots+k+1)\\=k^nd_0+H\cdot\frac{k^n-1}{k-1}\leq ((d-1)(d+2)+1)H k^n.\end{multline}
By assumption, $$\mathbf{F}(x)=\mathbf{A}(x)\mathbf{F}(x^k),$$ and hence for $n\geq 1$ we have $$\mathbf{F}(x)-\mathbf{\Phi}_n(x)=\mathbf{A}(x)\mathbf{A}(x^k)\cdots\mathbf{A}(x^{k^{n-1}})\left(\mathbf{F}(x^{k^n})-\mathbf{\Phi}_0(x^{k^n})\right).$$ 

Then for $n$ sufficiently large we have $t^{k^n}<\varepsilon$.  Hence if $e_i$ denotes the $d\times 1$ column vector whose $i$-th coordinate is $1$ and all other coordinates are zero, then 
 \begin{align*} \left|F_i(t)-P_{i,n}(t)/Q_n(t)\right|&=\left\|e_i^T\left(\mathbf{F}(t)-\mathbf{\Phi}_n(t)\right)\right\|\\
&=\left\|e_i^T\mathbf{A}(t)\mathbf{A}(t^k)\cdots\mathbf{A}(t^{k^{n-1}})(\mathbf{F}(t^{k^n})-\mathbf{\Phi}_0(t^{k^n}))\right\|\\
&\leq \left\|\left(\mathbf{F}(t^{k^n})-\mathbf{\Phi}_0(t^{k^n})\right)\right\|\cdot \prod_{\ell=0}^{n-1} \left\|\mathbf{A}(t^{k^{\ell}})\right\|\\
&\leq C\sqrt{d}{t^{d(d+2)Hk^n }}\cdot \prod_{\ell=0}^{n-1} \left\|\mathbf{A}(t^{k^{\ell}})\right\|.
\end{align*} Applying Lemma \ref{Abbound}, we see there is a positive constant $C_0=C_0(t)$ such that
$$\prod_{\ell=0}^{n-1} \left\|\mathbf{A}(t^{k^{\ell}})\right\|<C_0^n$$ for all $n\ge 1$ and hence
we have
$$\left|F_i(t)-P_{i,n}(t)/Q_n(t)\right| <C \sqrt{d} C_0^n t^{ d(d+2)Hk^n}$$ for all $i\in \{1,\ldots ,d\}$ and all $n$ sufficiently large. Also, applying Lemma \ref{Abbound}, in a neighbourhood of the origin, we see that there exists $\varepsilon>0$ such that the constant $C_0$ can be taken to be independent of $t$ (and dependent only on $\varepsilon$) for $t\in (0,\varepsilon)$.  To see that this gives the statement about the order of vanishing at $t=0$, note that if $F_i(t)-P_{i,n}(t)/Q_n(t)$ has a zero of order $\ell$ at $t=0$ then we can write $F_i(t)-P_{i,n}(t)/Q_n(t)$ as $t^{\ell}G(t)$ where $G(0)\neq 0$.  
It follows that there is a neighbourhood of zero such that $|t^{\ell}G(t)|>|G(0)||t|^{\ell}/2$ for $t$ in this neighbourhood.  Letting $t$ approach $0$ from the right and using the fact that
$$\left|F_i(t)-P_{i,n}(t)/Q_n(t)\right| <C \sqrt{d} C_0^n t^{ d(d+2)Hk^n}$$
gives $\ell\ge d(d+2)Hk^n$ and so $F_i(t)-P_{i,n}(t)/Q_n(t)$ has a zero at $t=0$ of order at least $d(d+2)Hk^n$.  The result follows.
\end{proof}
We now give the key result of this section, which shows that if $F_1(x),\ldots ,F_d(x)$ are power series and $a$ and $b$ are integers as in Notation \ref{notn: 1} then there are many good rational approximations to $F_i(a/b)$.

\begin{prop}\label{prop: lower} Assume the notation and assumptions of Notation \ref{notn: 1}. Then there exists a sequence of positive integers $\{q_n\}_{n\ge 0}$ such that:
\begin{enumerate}
\item[(i)] for $\delta=\frac{1}{3d^3}$, $i\in \{1,\ldots ,d\}$ and each $n\ge 0$, there is an integer $p_{i,n}$ such that $$|F_i(a/b)-p_{i,n}/q_n|<q_n^{-(1+\delta)}$$ for all $n$ sufficiently large;
\item[(ii)] there is a positive constant $C$ such that $q_n < q_{n+1} < C q_n^k$ for all $n$ sufficiently large.
\end{enumerate}
\end{prop}

\begin{proof}  
By assumption, $|a|= b^{\rho}$ for some $\rho<1/(d+1)$, and so by definition of $\delta$ it is easy to check that
\begin{equation}
\label{eq: rho}
(1-\rho)d(d+2)>(1+\delta)((d+2)(d-1)+1).
\end{equation} 
In particular, we take
\begin{equation}
\label{eq: rho2}
\varepsilon:=(1-\rho)d(d+2)-(1+\delta)((d+2)(d-1)+1)
\end{equation}
and note that $\varepsilon>0$.
By Lemma \ref{right}, we have that for each $n\ge 0$ there are polynomials $P_{1,n}(x),\ldots , P_{d,n}(x),Q_n(x)$ satisfying:
\begin{enumerate}
\item[(a)] $\max_{1\leq i\leq d}\{\deg P_{i,n}(x),\deg Q_n(x)\}\leq ((d+2)(d-1)+1)H k^n$;
\item[(b)] $Q_n(x)=B(x)B(x^k)\cdots B(x^{k^{n-1}})Q_0(x^{k^n})$;
\item[(c)] there exist positive constants $C_0$ and $C_1$ such that for $i\in \{1,\ldots ,d\}$ and for all $n$ sufficiently large we have $Q_n(a/b)\neq 0$ and 
$$\left|F_i(a/b)-P_{i,n}(a/b)/Q_n(a/b) \right|\leq C_1 C_0^n (a/b)^{ d(d+2)Hk^n}.$$ 
\end{enumerate}

Let $N_0\ge 0$ be such that $Q_n(a/b)\neq 0$ for $n\ge N_0$.  For $n< N_0$, we take $q_n=1$ and $p_{i,n}=1$ for $i\in \{1,\ldots ,d\}$.  For each $n\ge N_0$, we have $\max_{1\leq i\leq d}\{\deg P_{i,n}(x),\deg Q_n(x)\}\le  ((d+2)(d-1)+1)H k^n$, and hence we define integers
 \begin{equation} \label{q}
q_n:=b^{((d+2)(d-1)+1)H k^n} |Q_n(a/b)|\in \mathbb{Z}\end{equation}
 and
 \begin{equation} \label{p}
p_{i,n}:=b^{((d+2)(d-1)+1)Hk^n} P_{i,n}(a/b) \cdot \frac{Q_n(a/b)}{|Q_n(a/b)|}\in \mathbb{Z}.\end{equation}
Observe that
$$Q_n(a/b)=\left|Q_0((a/b)^{k^n})\prod_{\ell=0}^{n-1}B((a/b)^{k^{\ell}})\right|.$$
Since
$$\sum_{\ell\ge 0} \left(B\big((a/b)^{k^{\ell}}\big)-B(0)\right)$$ is absolutely convergent, $B(0)\neq 0$, and
$Q_0((a/b)^{k^n})\to Q_0(0)=1$ as $n\to\infty$, we see that there are positive constants $C_2$, $C_3$ such that
\begin{equation}
C_2 |B(0)|^n \le |Q_n(a/b)| \le C_3 |B(0)|^n
\end{equation}
for all $n\ge N_0$.
Thus
\begin{equation}\label{CqC}C_2 |B(0)|^n b^{((d+2)(d-1)+1)H k^n} \le q_n \le C_3 |B(0)|^n b^{((d+2)(d-1)+1)H k^n}\end{equation} for $n\ge N_0$.
Using the fact that $|B(0)|\ge 1$, we see that there is some $C>0$ such that
$$q_{n+1}\le C_3 |B(0)|^{n+1} b^{((d+2)(d-1)+1)H k^{n+1}} \le C q_n^k$$ for $n$ sufficiently large.
Finally, by (\ref{eq: rho}) and (\ref{eq: rho2}), we have
$$(1-\rho)d(d+2)Hk^n=(1+\delta)((d+2)(d-1)+1)Hk^n+\varepsilon Hk^n$$ for all $n$ and hence by inequality (\ref{CqC}), for $n$ sufficiently large,
we have 
\begin{align*}
\left|F_i(a/b)-P_{i,n}(a/b)/Q_n(a/b) \right| & \leq  C_1 C_0^n \left|a/b\right|^{ d(d+2)Hk^n}\\
& =   C_1 C_0^n b^{-(1-\rho)d(d+2)Hk^n} \\
&= C_1 C_0^n b^{-(1+\delta)((d+2)(d-1)+1)Hk^n} b^{-\varepsilon H k^n}\\
&\le C_1 C_0^n (C_3 |B(0)|^n)^{1+\delta}b^{-\varepsilon H k^n} q_n^{-(1+\delta)}\\
&<q_n^{-(1+\delta)},\end{align*}
for all $n$ sufficiently large.  The result now follows.
\end{proof}
We make the remark that the estimates obtained in the proof of Proposition \ref{prop: lower} do not imply that the $F_i(a/b)$ are irrational; they do, however, show that if one of $F_i(a/b)$ is rational then $p_{i,n}/q_{i,n}=F_i(a/b)$ for all $n$ sufficiently large.

\section{Lower bounds on rational approximations}\label{LB}

In this section, we continue the work we began in the preceding section and analyse the rational approximations to special values of functions satisfying a $k$-Mahler equation.  The purpose of this section is to show that these rational approximations cannot be too good.

In addition to the assumptions and notation of Notation \ref{notn: 1}, we adopt the following.

\begin{notn}\label{notn: 2} We adopt the following assumptions and notation.
\begin{enumerate}
\item We take $P_{1,n}(x),\ldots , P_{d,n}(x),Q_n(x)\in\mathbb{Z}[x]$ for $n\geq 0$ to be polynomials satisfying conditions $(i)$--$(iv)$ of Lemma \ref{right}.
\item We take the integers $p_{i,n}$ and $q_n$ to be as defined in Equations \eqref{q} and \eqref{p}.
\item We take $\mathbf{\Phi}_n(x):=[P_{1,n}(x)/Q_n(x),\ldots,P_{d,n}(x)/Q_n(x)]^T$ for $n\geq 0$.
\end{enumerate}
\end{notn}

\begin{lem}\label{SetI} Assume the notation and assumptions of Notation \ref{notn: 1} and let $N$ be a positive integer. If $c_1(x),\ldots,c_d(x)\in \mathbb{Q}[x]$ are polynomials of degree at most $N$, not all of which are zero, then $$\nu\left(\sum_{i=1}^d c_i(x)F_i(x)\right)\leq  2^{Hd} N k^{2d+3}.$$\end{lem}

\begin{proof} 
We define {\em Cartier operators} $\Lambda_i:\mathbb{Q}((x))\to \mathbb{Q}((x))$ by $$\Lambda_i\left(\sum_{n\geq -j}c_nx^n\right)=\sum_{n\geq (-j-i)/k} c_{kn+i}x^n$$ for $i=0,\ldots,k-1$.

Let $v_0:=\nu\left(\sum_{i=1}^d c_i(x)F_i(x)\right)$. We have $$\sum_{i=1}^d c_i(x)F_i(x)=[c_1(x),\ldots,c_d(x)]\mathbf{F}(x)=[c_1(x),\ldots,c_d(x)]\mathbf{A}(x)\mathbf{F}(x^k).$$ Note that $$[c_1(x),\ldots,c_d(x)]\mathbf{A}(x)=\left[\frac{c_{1,1}(x)}{B(x)},\ldots,\frac{c_{d,1}(x)}{B(x)}\right],$$ where $c_{1,1}(x),\ldots,c_{d,1}(x)$ are polynomials of degree at most $N+H$. It follows that \begin{equation}\label{J1} \nu\left(\sum_{i=1}^d c_{i,1}(x)F_i(x^k)\right)=v_0+\nu(B(x))=v_0,\end{equation} since $B(0)\neq 0$.

Let $j_0\in\{0,\ldots,k-1\}$ be such that $v_0\equiv j_0\ (\bmod\ k).$ Applying the Cartier operator $\Lambda_{j_0}$ to $\sum_{i=1}^d c_{i,1}(x)F_i(x^k)$ we have $$\nu\left(\sum_{i=1}^d \Lambda_{j_0}\left(c_{i,1}(x)\right)F_i(x)\right)=(v_0-j_0)/k.$$ Set $v_1:=(v_0-j_0)/k$. Then \begin{equation}\label{J3} v_0/k-1\leq v_1\leq v_0/k.\end{equation} Moreover, $\Lambda_{j_0}\left(c_{1,1}(x)\right),\ldots,\Lambda_{j_0}\left(c_{d,1}(x)\right)$ have degree at most $N_1:=(N+H)/k.$ Similarly, $$\left[\Lambda_{j_0}\left(c_{1,1}(x)\right),\ldots,\Lambda_{j_0}\left(c_{d,1}(x)\right)\right]\mathbf{A}(x)=\left[\frac{c_{1,2}(x)}{B(x)},\ldots,\frac{c_{d,2}(x)}{B(x)}\right],$$ for some polynomials $c_{1,2}(x),\ldots,c_{d,2}(x)$ of degree at most $N_1+H$ and we have $$\nu\left(\sum_{i=1}^d c_{i,2}(x)F_i(x^k)\right)=v_1.$$

Let $j_1\in\{0,\ldots,k-1\}$ be such that $v_1\equiv j_1\ (\bmod\ k).$ Applying the operator $\Lambda_{j_1}$ to $\sum_{i=1}^d c_{i,2}(x)F_i(x^k)$ we see that $$\nu\left(\sum_{i=1}^d \Lambda_{j_1}\left(c_{i,2}(x)\right)F_i(x)\right)=(v_1-j_1)/k.$$ Now set $v_2:=(v_1-j_1)/k$. Then \begin{equation}\label{J4} v_1/k-1\leq v_2\leq v_1/k\end{equation} and thus by \eqref{J3} we have \begin{equation}\label{J5} v_0/k^2-1/{k}-1\leq v_2\leq v_0/k^2.\end{equation} Moreover, $\Lambda_{j_1}\left(c_{1,2}(x)\right),\ldots,\Lambda_{j_1}\left(c_{d,2}(x)\right)$ have degree at most $$N_2:=(N_1+H)/{k}\leq N/k^2+H/k^2+H/k.$$

Continuing in this manner, for each $n$, we can construct polynomials $c_{1,n}(x),\ldots,$ $c_{d,n}(x)$ such that for some $j\in\{0,\ldots,k-1\}$ we have $$\nu\left(\sum_{i=1}^d \Lambda_j\left(c_{i,n}(x)\right)F_i(x)\right)=v_n,$$ where \begin{equation}\label{J6}v_0/{k^n}-k/{(k-1)}\leq v_n\leq v_0/k^n,\end{equation} and the degrees of $$\left[b_{1,n}(x),\ldots,b_{d,n}(x)\right]:=\left[\Lambda_j\left(c_{1,n}(x)\right),\ldots,\Lambda_j\left(c_{d,n}(x)\right)\right]$$ are at most $$N/k^n+H/k+\cdots+H/k^n\leq N/k^n+H/(k-1).$$

Let $n_0$ be the unique natural number such that $k^{n_0}>N\geq k^{n_0-1}$. Then for $$n\in\{n_0,n_0+1,\ldots,n_0+\lfloor H/(k-1)+2\rfloor d\}$$ the constructed vectors $[b_{1,n}(x),\ldots,b_{d,n}(x)]$ are in the vector space of $d$-tuples of polynomials of degree less than $H/(k-1)+1$, which has dimension at most $\lfloor H/(k-1)+2\rfloor d$ as a $\mathbb{Q}$-vector space. Hence for $n\in\{n_0,n_0+1,\ldots,n_0+\lfloor H/(k-1)+2\rfloor d\}$ the vectors $[b_{1,n}(x),\ldots,b_{d,n}(x)]$ are linearly dependent over $\mathbb{Q}$.  In particular, there exist constants $\lambda_{n}$, not all of which are zero, such that $$\sum_{n=n_0}^{n_0+\lfloor H/(k-1)+2\rfloor d} \lambda_{n}\cdot [b_{1,n}(x),\ldots,b_{d,n}(x)]=0.$$ But this gives that \begin{align*} 0 &= \sum_{n=n_0}^{n_0+\lfloor H/(k-1)+2\rfloor d} \lambda_{n}\cdot [b_{1,n}(x),\ldots,b_{d,n}(x)]\cdot\mathbf{F}(x)\\ &=\sum_{n=n_0}^{n_0+\lfloor H/(k-1)+2\rfloor d} \lambda_{n}\cdot \sum_{i=1}^d b_{i,n}(x)F_i(x),\end{align*} which immediately yields that the functions $$G_n(x):=\sum_{i=1}^d b_{i,n}(x)F_i(x)$$ for $n\in\{n_0,n_0+1,\ldots,n_0+\lfloor H/(k-1)+2\rfloor d\}$ are linearly dependent over $\mathbb{Q}$.

Suppose, towards a contradiction, that $$v_0> \frac{k^2}{(k-1)^2}k^{n_0+(H/(k-1)+2)d}.$$ Then for $i\in\{0,\ldots,\lfloor H/(k-1)+2\rfloor\cdot d-1\}$ we have 
by Inequality \eqref{J6} that \begin{equation*}\label{Jp7} v_{n_0+i}-v_{n_0+i+1} 
\ge \frac{v_0}{k^{n_0+i}}-\frac{k}{k-1} - \frac{v_0}{k^{n_0+i+1}} = v_0k^{-n_0-i}\left(1-\frac{1}{k}\right) - \frac{k}{k-1}>0.
\end{equation*}
Thus the numbers $$v_{n_0},v_{n_0+1},\ldots, v_{n_0+\lfloor H/(k-1)+2\rfloor d}$$ are all distinct, and so for $n\in\{n_0,n_0+1,\ldots,n_0+\lfloor H/(k-1)+2\rfloor d\}$ the functions $G_n(x)$ are linearly independent over $\mathbb{Q}$, as we have $\nu(G_n(x))=v_n$ for all $n$.
This is a contradiction. 

Hence
$$v_0\le \frac{k^2}{(k-1)^2}k^{n_0+(H/(k-1)+2)d} \le \frac{k^2}{(k-1)^2}\cdot (kN) k^{\frac{Hd}{k-1}}k^{2d}.$$
Note that $k^{1/(k-1)}\le 2$ and thus
$$v_0\le 2^{Hd} N k^{2d+3}.$$
The result follows.
\end{proof}

\begin{lem}\label{Lem: LIi} Assume the assumptions and notation of Notations \ref{notn: 1} and \ref{notn: 2} and let $m$ be a natural number with the property that $k^m\ge 2^{Hd}k^{2d+3}$.  Then for each $n\geq 0$, we have that $\mathbf{\Phi}_n(x),\mathbf{\Phi}_{n+m}(x),\ldots,\mathbf{\Phi}_{n+m(d-1)}(x)$ are linearly independent over $\mathbb{Q}(x)$.
\end{lem}

\begin{proof} We prove this by induction on $n\geq 0$. We first prove the base case $n=0$ by showing by induction on $i$ that $\mathbf{\Phi}_0(x),\mathbf{\Phi}_m(x),\ldots,\mathbf{\Phi}_{mi}(x)$ are linearly independent over $\mathbb{Q}(x)$ for all $i<d$. When $i=0$, $\mathbf{\Phi}_0(x)$ is linearly independent over $\mathbb{Q}(x)$ as it is nonzero.

Suppose that $\mathbf{\Phi}_0(x),\mathbf{\Phi}_m(x),\ldots,\mathbf{\Phi}_{mi}(x)$ are linearly independent over $\mathbb{Q}(x)$ for all $i$ with $i<j<d$ and consider the case $i=j$. Then $\mathbf{\Phi}_0(x),\mathbf{\Phi}_m(x),\ldots,$ $\mathbf{\Phi}_{m(j-1)}(x)$ are linearly independent over $\mathbb{Q}(x)$.

If $\mathbf{\Phi}_0(x),\mathbf{\Phi}_m(x),\ldots,\mathbf{\Phi}_{mj}(x)$ are linearly dependent over $\mathbb{Q}(x)$ then the column spaces of the matrices $(\mathbf{\Phi}_0(x),\mathbf{\Phi}_m(x),\ldots,\mathbf{\Phi}_{mj}(x))$ and $(\mathbf{\Phi}_0(x),\mathbf{\Phi}_m(x),\ldots,$ $\mathbf{\Phi}_{m(j-1)}(x))$ are the same, and hence if $b_1(x),\ldots,b_d(x)\in \mathbb{Q}(x)$ are such that  \begin{equation}\label{JII3}[b_1(x),\ldots,b_d(x)](\mathbf{\Phi}_0(x),\mathbf{\Phi}_m(x),\ldots,\mathbf{\Phi}_{m(j-1)}(x))=0,\end{equation} then we have $$[b_1(x),\ldots,b_d(x)]\mathbf{\Phi}_{mj}(x)=0.$$ By Cramer's rule, there is a nonzero solution to \eqref{JII3} with $b_i(x)$ equal---up to sign---to the determinant of a $j\times j$ submatrix of $(\mathbf{\Phi}_0(x),\mathbf{\Phi}_m(x),\ldots,$ $\mathbf{\Phi}_{m(j-1)}(x))$ for $i=1,\ldots, j+1$ and $b_i(x)=0$ for $i=j+2,\ldots,d$. 
By Lemma \ref{right}, if 
we define $d_n$ to be $\max_{1\le i\le d} \left\{ \deg(P_{i,n}(x)), \deg(Q_n(x))\right\}$ 
and define $h_{n}$ to be $\min_{1\leq i\leq d}\{\nu(F_i(x)-P_{i,n}(x)/Q_n(x))\}$ for $n\geq 0$, then we have
\begin{equation}\label{eq: hhh}
h_n \ge d(d+2)Hk^n
\end{equation}
and 
 \begin{equation}
\label{eq: ddd}
d_n\le ((d+2)(d-1)+1)Hk^n
\end{equation} for all $n\ge 0$.
Observe that each determinant of a $j\times j$ submatrix of $(\mathbf{\Phi}_0(x),\mathbf{\Phi}_m(x),\ldots,$ $\mathbf{\Phi}_{m(j-1)}(x))$ is a rational function with numerator and denominator of degree at most $d_0+d_m+\cdots+d_{m(j-1)}$; moreover, they all have a common denominator of at most this degree. Thus there exist polynomials $b_1(x),\ldots,b_d(x)\in \mathbb{Q}[x]$, not all zero, of degree at most $d_0+d_m+\cdots+d_{m(j-1)}$ such that $$[b_1(x),\ldots,b_d(x)](\mathbf{\Phi}_0(x),\mathbf{\Phi}_m(x),\ldots,\mathbf{\Phi}_{m(j-1)}(x))=0,$$ which gives $$[b_1(x),\ldots,b_d(x)]\mathbf{\Phi}_{mj}(x)=0.$$ 

Thus \begin{align}\nonumber \nu\left(\sum_{i=1}^d b_i(x)F_i(x)\right)&=\nu\left(\sum_{i=1}^d b_i(x)F_i(x)-\sum_{i=1}^d b_i(x)P_{i,jm}(x)/Q_{jm}(x)\right)\\
\nonumber &=\nu\left(\sum_{i=1}^d b_i(x)(F_i(x)-P_{i,jm}(x)/Q_{jm}(x))\right)\\
\label{pFhjm}&\geq h_{jm} \ge d(d+2)Hk^{jm}.\end{align} 

But the polynomials $b_i(x)$ are of degree at most $d_0+d_m+\cdots+d_{m(j-1)}$ and so by Lemma \ref{SetI} and (\ref{eq: ddd}) we have
\begin{align*}  d(d+2)Hk^{jm} \le \nu\left(\sum_{i=1}^d b_i(x)F_i(x)\right)  &\le  2^{Hd} (d_0+d_m+\cdots +d_{m(j-1)}) k^{2d+3} \\
&\le 2^{Hd} ((d+2)(d-1)+1)H \frac{k^{jm}}{(k^m-1)}k^{2d+3}.   
 \end{align*}
Dividing this inequality by $d(d+2)Hk^{jm}$ we get
$$k^m \le  1+ 2^{Hd}\frac{((d+2)(d-1)+1)}{d(d+2)}k^{2d+3}<2^{Hd}k^{2d+3},$$
a contradiction.

This completes the base case for our induction on $n$.

Next assume that the result is true for all natural numbers less than or equal to $n$. Notice by construction $\mathbf{\Phi}_{n+1}(x)=\mathbf{A}(x)\mathbf{\Phi}_{n}(x^k),$ and $\mathbf{A}(x)$ is invertible. Thus we have that \begin{multline}\label{PAP} \left(\mathbf{\Phi}_n(x),\mathbf{\Phi}_{n+m}(x),\ldots,\mathbf{\Phi}_{n+m(d-1)}(x)\right)\\ =\mathbf{A}(x)\left(\mathbf{\Phi}_{n-1}(x^k),\mathbf{\Phi}_{n-1+m}(x^k),\ldots,\mathbf{\Phi}_{n-1+m(d-1)}(x^k)\right).\end{multline} By the inductive hypothesis, the matrix $$\left(\mathbf{\Phi}_{n-1}(x^k),\mathbf{\Phi}_{n-1+m}(x^k),\ldots,\mathbf{\Phi}_{n-1+m(d-1)}(x^k)\right)$$ is invertible. Since $\mathbf{A}(x)$ is also invertible, the product $$\mathbf{A}(x)\left(\mathbf{\Phi}_{n-1}(x^k),\mathbf{\Phi}_{n-1+m}(x^k),\ldots,\mathbf{\Phi}_{n-1+m(d-1)}(x^k)\right)$$ is invertible, and thus the equality \eqref{PAP} gives the desired result.  This proves the lemma.
\end{proof}

We will need a few more estimates for the norms of matrices. 

\begin{lem}\label{Binvbound} Let $d$ and $H$ be natural numbers and let $\mathbf{B}(x)$ be a $d\times d$ invertible matrix whose $(i,j)$-entry is the rational function $b_{i,j}(x)/b(x)$ for $1\le i,j\le d$, where $b_{i,j}(x),b(x)$ are polynomials of degree at most $H$. Then there exist $\varepsilon>0$ and a positive constant $C$ such that for all $\alpha\in (-\varepsilon,\varepsilon)\setminus \{0\}$ we have $$\left\|\mathbf{B}(\alpha)^{-1}\right\|\leq \frac{C}{|\alpha|^{Hd}}.$$ 
\end{lem}

\begin{proof}  
Define $\mathbf{D}(x):=b(x)\mathbf{B}(x)$, and let $\Delta_{i,j}(x)$ denote the determinant of the $(d-1)\times (d-1)$ submatrix of $\mathbf{D}(x)$ formed by removing the $i$-th row and $j$-th column of $\mathbf{D}(x)$. Note that $\Delta_{i,j}(x)$ is a polynomial of degree at most $H(d-1)$. The classical adjoint of the square matrix $\mathbf{B}(x)$ is the $d\times d$ matrix \begin{align*}\adj(\mathbf{B}(x))&=b(x)^{-(d-1)}\left((-1)^{i+j}\Delta_{j,i}(x)\right)_{1\leq i,j\leq d}.
\end{align*} Note that each entry of $\left((-1)^{i+j}\Delta_{j,i}(x)\right)_{1\leq i,j\leq d}$ is a polynomial of degree at most $H(d-1)$.

Since $$\mathbf{B}(x)^{-1}=\frac{\adj(\mathbf{B}(x))}{\det(\mathbf{B}(x))},$$ and $$\det(\mathbf{B}(x))=\det\left(b(x)^{-1}\cdot \mathbf{D}(x)\right)=\det(\mathbf{D}(x))b(x)^{-d},$$ we have that $$\mathbf{B}(x)^{-1}=\frac{b(x)}{\det(\mathbf{D}(x))}\cdot \left((-1)^{i+j}\Delta_{j,i}(x)\right)_{1\leq i,j\leq d}.$$ Note that the entries of $\mathbf{D}(x)$ are all polynomials of degree at most $H$, and so $\det(\mathbf{D}(x))$ is a polynomial of degree at most $Hd$.  In particular, there is some $\varepsilon>0$ such that 
$\det(\mathbf{D}(x))$ has no zeros on $\{x\in \mathbb{C}~:~0<|x|\le \varepsilon\}$ and so
$x^{Hd}/\det(\mathbf{D}(x))$ is continuous on the compact set $\{x\in \mathbb{C}~:~|x|\le \varepsilon\}$.  In particular, there is some constant $C_0>0$ such that
$$|x^{Hd}/\det(\mathbf{D}(x))|\le C_0$$ for $|x|\le \varepsilon$.  Similarly, there is a positive constant $C_1$ such that $|b(x)\Delta_{i,j}(x)|\le C_1$ for $|x|\le \varepsilon$ and for $1\le i,j\le d$.

Thus for $\alpha\in (-\varepsilon,\varepsilon)\setminus\{0\}$ we have
 $$\left\|\mathbf{B}(\alpha)^{-1}\right\|\leq \left(\sum_{1\leq i,j\leq d}\left|\frac{b(\alpha)\Delta_{i,j}(\alpha)}{\det(\mathbf{D}(\alpha))}\right|^2\right)^{1/2} \le (d^2 C_1^2C_0^2 |\alpha|^{-2Hd})^{1/2}.$$ 
The result follows taking $C=dC_1 C_0$. 
\end{proof}

Applying the preceding results gives the following lemma.  As is customary, when working with a vector space $K^n$ over a field $K$, we let $e_i$ denote the $1\times n$ row vector with a $1$ in the $i$-th coordinate and zeros in the remaining coordinates.

\begin{lem}\label{Lemma1} Assume the assumptions and notation of Notations \ref{notn: 1} and \ref{notn: 2} and let $m$ be a natural number with the property that $k^m\ge 2^{Hd}k^{2d+3}$.  For each natural number $n\geq 1$, let $\mathbf{C}_n(x)$ denote the $d\times (d-1)$ matrix whose $j$-th column is ${\bf \Phi}_{n}(x)-{\bf \Phi}_{n+jm}(x)$.  Then $\mathbf{C}_n(x)$ has rank $d-1$.  Moreover, there are positive constants $c_0$ and $c_1$ such that for $i\in \{2,\ldots, d\}$ we either have $$\|e_i^T{\bf C}_n(a/b)\|\geq c_0 c_1^n |a/b|^{(d-1)(d^2+d)Hk^{m(d-1)}k^n}.$$ for all n sufficiently large or $\|e_i^T\mathbf{C}_n\left(a/b\right)\|$ is eventually zero.
\end{lem}

\begin{proof} For the first part of the lemma, suppose to the contrary that ${\bf C}_n(x)$ has rank strictly less than $d-1$. Then the vectors $\{\mathbf{\Phi}_n(x)-\mathbf{\Phi}_{n+im}(x)\}_{i=1}^{d-1}$ are linearly dependent. Thus there exist $a_1(x),\ldots,a_{d-1}(x)\in \mathbb{Q}(x)$, not all zero, such that $$\sum_{i=1}^{d-1}a_i(x)(\mathbf{\Phi}_n(x)-\mathbf{\Phi}_{n+im}(x))=0.$$ Hence $$(a_1(x)+a_2(x)\cdots+a_{d-1}(x))\mathbf{\Phi}_n(x)-\sum_{i=1}^{d-1}a_i(x)\mathbf{\Phi}_{n+im}(x)=0,$$ and so the set of vectors $\{\mathbf{\Phi}_{n+im}(x)\}_{i=0}^{d-1}$ is linearly dependent over $\mathbb{Q}(x)$.  But this contradiction the conclusion of Lemma \ref{Lem: LIi}. 

Recall from Lemma \ref{right} $(v)$ that $P_{1,n}(x)/Q_n(x)=1$ for all $n$ and hence the top row of ${\bf C}_n(x)$ is zero.  We let ${\bf C}_n^{(1)}(x)$ denote the $(d-1)\times (d-1)$ matrix obtained by deleting the first row of ${\bf C}_n(x)$.  Then ${\bf C}_n^{(1)}(x)$ must have rank $(d-1)$ and thus it is invertible.

Then Lemma \ref{right} $(ii)$ gives 
\begin{equation}
\label{eq: BBB}
{\bf C}_n(x)=\mathbf{A}(x)\mathbf{A}(x^k)\cdots \mathbf{A}(x^{k^{n-1}}){\bf C}_0(x^{k^n})
\end{equation} for each $n\ge 1$.  
Let $i\in \{2,\ldots ,d\}$.  By construction we have
$$\|e_{i-1}^T{\bf C}^{(1)}_n(\alpha)\| = \|e_i^T{\bf C}_n(\alpha)\|$$ for $\alpha\in (-1,1)$ at which ${\bf C}_n(x)$ is defined.

From the first part of the lemma we have that ${\bf C}_0^{(1)}(x)$ has nonzero determinant and by Lemma \ref{right}, its entries are rational functions with numerator and common denominator having degree bounded by $((d-1)(d+2)+1)Hk^{m(d-1)}$.  Hence by Lemma \ref{Binvbound}, there is some $\varepsilon>0$ and some constant $C_0>0$ such that $$\|{\bf C}_0^{(1)}(\alpha)^{-1}\| < C_0|\alpha|^{-(d-1)((d-1)(d+2)+1)Hk^{m(d-1)}}$$ for $\alpha\in (-\varepsilon,\varepsilon)\setminus \{0\}$.  

Since $F_1(x)=1$ we have
$$1=e_1^T{\bf F}(x)=e_1^T{\bf A}(x){\bf F}(x^k).$$  Since $F_1,\ldots ,F_d$ are linearly independent over $\mathbb{Q}(x)$, we see that the first row of ${\bf A}(x)$ is $e_1^T$.  Let ${\bf B}(x)$ denote the $(d-1)\times (d-1)$ submatrix of ${\bf A}(x)$ obtained by deleting the first row and first column.  This fact along with  (\ref{eq: BBB}) imply that  
\begin{equation}\label{cbbc0}\mathbf{C}_n^{(1)}(a/b)=\mathbf{B}(a/b)\mathbf{B}((a/b)^k)\cdots \mathbf{B}((a/b)^{k^{n-1}})\mathbf{C}_0^{(1)}((a/b)^{k^n}).\end{equation}

By Notation \ref{notn: 1} the entries of ${\bf B}(x)$ are rational functions with numerator and common denominator of degree at most $H$ and so by Lemma \ref{Binvbound} there exists some $\varepsilon_0\in (0,\varepsilon)$ such that 
\begin{equation}
\label{eq: Bb}
\|{\bf B}(\alpha)^{-1}\|\le C_1|\alpha|^{-H(d-1)}
\end{equation} for $\alpha\in (-\varepsilon_0,\varepsilon_0)\setminus \{0\}$.

Pick a natural number $r$ such that $(a/b)^{k^n}\in (-\varepsilon_0,\varepsilon_0)$ for $n\ge r$.  
Then \eqref{cbbc0} and the definition of $\mathbf{B}(x)$ imply that for $i\in\{2,\ldots,d\}$ we have \begin{multline}\label{CBBC}\left\|e_i^T\mathbf{C}_n(a/b)\right\|\geq \left\|e_i^T\mathbf{A}(a/b)\mathbf{A}((a/b)^k)\cdots \mathbf{A}((a/b)^{k^{r-1}})\right\|\\ \times\left(\left\|\mathbf{C}_0^{(1)}((a/b)^{k^n})^{-1}\right\|\cdot\prod_{\ell=r}^{n-1}\left\|\mathbf{B}((a/b)^{k^{\ell}})^{-1}\right\|\right)^{-1}.\end{multline}

If $$e_i^T\mathbf{A}(a/b)\mathbf{A}((a/b)^k)\cdots \mathbf{A}((a/b)^{k^{r-1}})=0$$ then $e_i^T{\bf C}_n(a/b)$ is zero for all $n\ge r$ and so we may assume that 
$$C_2:=\|e_i^T\mathbf{A}(a/b)\mathbf{A}((a/b)^k)\cdots \mathbf{A}((a/b)^{k^{r-1}})\|>0.$$ 

Inequality \eqref{eq: Bb} gives \begin{align}\nonumber \left(\prod_{\ell=r}^{n-1}\left\|\mathbf{B}((a/b)^{k^{\ell}})^{-1}\right\|\right)^{-1}&\geq  C_1^{r-n}|a/b|^{H(d-1)k^{n-1}}\cdots
|a/b|^{H(d-1)k^{r}} \\
\label{k3abHH}&\geq C_1^{r-n} |a/b|^{H(d-1)k^{n}/(k-1)}.
\end{align}

Hence by (\ref{CBBC}) we have
\begin{equation}
\label{eq: CBBC}\|e_i^T{\bf C}_n(a/b)\| \ge C_2C_0^{-1}C_1^{r-n}
|a/b|^{(d-1)((d-1)(d+2)+1)Hk^{m(d-1)}k^n}
 |a/b|^{H(d-1)k^{n}/(k-1)}.
 \end{equation}
 Note that 
 $$((d-1)(d+2)+1)k^{m(d-1)}+1/(k-1)<(d^2+d)k^{m(d-1)}$$ and so taking
 $c_0=C_2 C_0^{-1}C_1^r$ and $c_1=C_1^{-1}$ in (\ref{eq: CBBC}) gives
 $$\|e_i^T{\bf C}_n(a/b)\|\geq c_0 c_1^n
|a/b|^{(d-1)(d^2+d)Hk^{m(d-1)}k^n}.$$  The result follows.
\end{proof}

\begin{defn} We call an infinite subset $S$ of the natural numbers {\em syndetic} if there exists a natural number $m$ such that for each $s\in S$ there exists some $t\in S$ with $0<t-s\le m$.  \end{defn}

 We have the following lemma.

\begin{lem}\label{Lemma2} Assume the assumptions and notation of Notations \ref{notn: 1} and \ref{notn: 2} and let $m$ be the smallest natural number with the property that $k^m\ge 2^{Hd}k^{2d+3}$.  If $i\in \{1,\ldots ,d\}$, then either $F_i(a/b)$ is a rational number or $$|F_i(a/b)-P_{i,n}(a/b)/Q_n(a/b)|\geq {b^{-(1-\rho)d^3Hk^{m(d-1)}k^{n}}}$$ for a syndetic set of natural numbers $n$ with gaps eventually bounded by at most $(d-1)m$.
\end{lem}

\begin{proof} Let $L:=d^3Hk^{m(d-1)}$. If the conclusion of the statement of the theorem is not true, then there are infinitely many natural numbers $N$ such that $$|F_i(a/b)-P_{i,n}(a/b)/Q_n(a/b)|<b^{-(1-\rho)Lk^n}$$ for $n=N,N+m,\ldots,N+(d-1)m$.  So by triangle inequality $$|P_{i,N}(a/b)/Q_N(a/b)-P_{i,N+jm}(a/b)/Q_{N+jm}(a/b)|<2{b^{-(1-\rho)Lk^N}}$$ for $j=1,\ldots,d-1.$

As defined in the previous lemma, let $\mathbf{C}_n(x)$ denote the $d\times (d-1)$ matrix whose $j$-th column is ${\bf \Phi}_{n}(x)-{\bf \Phi}_{n+jm}(x)$.
Then 
$$|P_{i,N}(a/b)/Q_N(a/b)-P_{i,N+jm}(a/b)/Q_{N+jm}(a/b)|$$ is the $j$-th coordinate of
$e_i^T\mathbf{C}_N(a/b)$ and so
we have $$\left\|e_i^T\mathbf{C}_N(a/b)\right\|<\frac{2\sqrt{d-1}}{b^{(1-\rho)Lk^N}}.$$ 
But by Lemma \ref{Lemma1}, we have that either $e_i^T\mathbf{C}_n(a/b)$ is zero for all $n$ sufficiently large or there exist positive constants $c_0$ and $c_1$ such that
$$\|e_i^T\mathbf{C}_n(a/b)\|\ge  c_0 c_1^n
(a/b)^{(d-1)(d^2+d)Hk^{m(d-1)}k^n}$$ for $n$ sufficiently large.
If the former possibility occurs then there is some $r>0$ such that
$$|P_{i,r}(a/b)/Q_r(a/b)-P_{i,r+nm}(a/b)/Q_{r+nm}(a/b)|=0$$ for all $n\ge 0$ sufficiently large, and since
$P_{i,n}(a/b)/Q_n(a/b)\to F_i(a/b)$ as $n\to \infty$, we see $F_i(a/b)=P_{i,r}(a/b)/Q_r(a/b)\in \mathbb{Q}$.

If the latter possibility occurs then we have
\begin{equation}\label{eq: C0C1}
c_0c_1^n b^{-(1-\rho)(d-1)(d^2+d)Hk^{m(d-1)+n}}<2(d-1)^{1/2}{b^{-(1-\rho)Lk^n}}
\end{equation}
for infinitely many $n$.  
Taking logarithms of both sides, dividing by $(1-\rho)Hk^{n+m(d-1)}$, and taking the limit supremum of both sides over all $n$ for which the inequality (\ref{eq: C0C1}) holds, we see that 
$$(d-1)(d^2+d)\ge d^3,$$ a contradiction.  The result follows.
\end{proof}

\begin{prop}\label{prop: upper} Assume the assumptions and notation of Notations \ref{notn: 1} and \ref{notn: 2}, let $i\in \{2,\ldots ,d\}$, and let $m$ be the smallest natural number with the property that $k^m\ge 2^{Hd}k^{2d+3}$.  Then either $F_i(a/b)$ is a rational number or for each sufficiently large natural number $n$, there is $j\in\{n,n+1,\ldots,n+m(d-1)\}$ such that $$\left|F_i(a/b)-\frac{p_{i,j}}{q_j}\right|\geq 1/q_{j}^{(1-\rho)dk^{m(d-1)}}.$$
\end{prop}

\begin{proof} Suppose that $F_i(a/b)$ is irrational.  From  (\ref{q}), we have 
 \begin{equation} \label{q1}
q_n:=b^{((d+2)(d-1)+1)H k^n} |Q_n(a/b)|\in \mathbb{Z}.\end{equation}
Equation
 \eqref{CqC} shows there exist positive constants $C_2$ and $C_3$ that \begin{equation}\label{Cqee} C_2 |B(0)|^n b^{((d+2)(d-1)+1)H k^n}  \le q_n \le 
 C_3 |B(0)|^n b^{((d+2)(d-1)+1)H k^n} 
 \end{equation} 
 for all $n$ sufficiently large.
 Since $p_{i,n}/q_{n}=P_n(a/b)/Q_n(a/b)$ we see from Lemma~\ref{Lemma2} that for each $n$ sufficiently large there is some
 $j\in\{n,n+1,\ldots,n+m(d-1)\}$ such that
 $$\left|F_i(a/b)-\frac{p_{i,j}}{q_j}\right| \geq   {b^{-(1-\rho)d^3Hk^{m(d-1)}k^{j}}}.$$
 For $n$ sufficiently large we have
 $${b^{-(1-\rho)d^3Hk^{m(d-1)}k^{j}}}\ge  q_j^{-(1-\rho)dk^{m(d-1)}}$$ by inequality \eqref{Cqee}.
This proves the proposition.
\end{proof}

\section{Mahler numbers are not Liouville}\label{mu}

In this section, combining Propositions \ref{prop: lower} and \ref{prop: upper}, we give a quantitative bound on the irrationality exponent of a Mahler number assuming Mahler's condition is satisfied. That is, we prove the following result. 

\begin{thm}\label{main} Suppose that $F(x)\in \mathbb{Z}[[x]]$ satisfies the Mahler equation \eqref{MFE}, that $a/b$ is inside the radius of convergence of $F(x)$, $\log|a|/\log b\in [0,1/(d+2))$, and $a_0((a/b)^{k^n}) \neq 0$ for $n\ge 0$. Then 
$$\mu(F(a/b))\le 4^{H(k^d+1)d^2}k^{5d^2},$$ where $H=\max_{0\leq i\leq d}\{\deg a_i(x)\}.$ 
\end{thm}

We use a slightly modified version of the following classical lemma.

\begin{lem}\label{ARLem} Let $\xi, \delta, \varrho$ and 
$\vartheta$ be 
real numbers such that $0<\delta\leqslant\varrho$ and $\vartheta\geqslant 1$. Let us assume that there exists 
a sequence $\{p_n/q_n\}_{n\geqslant 1}$ of rational numbers and some positive constants 
$c_0,c_1$ and $c_2$ such that both $$q_n<q_{n+1}\leqslant c_0 q_n^{\vartheta},$$ and $$\frac{c_1}
{q_n^{1+\varrho}}\leqslant\left|\xi-\frac{p_n}{q_n}\right|\leqslant\frac{c_2}{q_n^{1+\delta}}.$$ Then we 
have that $$\mu(\xi)\leqslant (1+\varrho)\frac{\vartheta}{\delta}.$$
\end{lem}
\begin{proof} See Adamczewski and Rivoal \cite[Lemma 4.1]{AR2009}. 
\end{proof}
As a consequence of this lemma, we have the following result.

\begin{lem}\label{ARLemSyn} Let $\xi, \delta, \varrho$ and $\vartheta$ be real numbers such that $0<\delta\leqslant\varrho$ and $\vartheta\geqslant 1$ and let $d\in\mathbb{N}$. Let us assume that there exists a sequence $\{p_n/q_n\}_{n\geqslant 1}$ of rational numbers and some positive constants $c_3,c_4$ and $c_5$ such that for all $n\geq 1$ we have
\begin{itemize}
\item[$(i)$] $0<q_n<q_{n+1}\leq c_3q_n^{\vartheta}$, 
\item[$(ii)$] $\left|\xi-\frac{p_n}{q_n}\right|\leqslant\frac{c_4}{q_n^{1+\delta}},$ 
\item[$(iii)$] and that there is a syndetic subset $S$ of $\mathbb{N}$ with gaps eventually bounded by $\ell$ such that \begin{equation}\label{ARleft}\frac{c_5}{q_n^{1+\varrho}}\leqslant\left|\xi-\frac{p_n}{q_n}\right|\end{equation} for all $n\in S$.
\end{itemize} 
Then $$\mu(\xi)\leqslant (1+\varrho)\frac{\vartheta^\ell}{\delta}.$$
\end{lem}

\begin{proof} Let $S$ be the set described in assumption $(iii)$ of the lemma, and let $S'$ be an infinite subset of $S$ for which the difference of consecutive elements is at most $\ell$. Set $s_1:=\min\{k\in S'\}$ and for $n\geqslant 2$ set $$s_n:=\min\left\{k\in(s_{n-1},s_{n-1}+\ell]: \frac{c_5}
{q_n^{1+\varrho}}\leqslant\left|\xi-\frac{p_n}{q_n}\right|\right\}.$$ By assumption $(iii)$ the number $s_n$ exists and for $n\geqslant 2$ we have $s_n\leqslant s_{n-1}+\ell.$ Define the 
sequence $\{Q_n\}_{n\geqslant 1}$ by $Q_n=q_{s_n}$ for all $n\geqslant 1$. By assumption $(i)$ 
we have that $$0<q_{s_n}<q_{s_{n+1}}\leqslant q_{s_n+\ell}\leqslant c_3q_{s_n+\ell-1}^{\vartheta}\leqslant c_3^2
q_{s_n+\ell-2}^{\vartheta^2}\leqslant\cdots\leqslant c_3^\ell q_{s_n+\ell-\ell}^{\vartheta^\ell}=c_3^\ell q_{s_n}^{\vartheta^\ell},$$ so that 
\begin{equation}\label{AR1}0<Q_n<Q_{n+1}\leq C_0 Q_n^{\vartheta^\ell},\end{equation} where $C_0=c_3^\ell$ is a positive constant independent of $n$. Also, we have for all 
$s_n$ that $$\frac{c_5}{q_{s_n}^{1+\varrho}}\leqslant\left|\xi-\frac{p_{s_n}}{q_{s_n}}\right|\leqslant
\frac{c_4}{q_{s_n}^{1+\delta}},$$  so that setting $P_n=p_{s_n}$, we have that \begin{equation}
\label{AR2}\frac{c_5}{Q_n^{1+\varrho}}\leqslant\left|\xi-\frac{P_n}{Q_n}\right|\leqslant\frac{c_4}
{Q_n^{1+\delta}}.\end{equation} We now apply Lemma \ref{ARLem} to the sequence $\{P_n/Q_n\}_{n
\geqslant 1}$ of rational numbers using \eqref{AR1} and \eqref{AR2} to give that $$\mu(\xi)\leqslant 
(1+\varrho)\frac{\vartheta^\ell}{\delta}.$$ This proves the lemma.
\end{proof}

We are now ready to prove one of the main results of this section.

\begin{thm}\label{maineffective} Assume the assumptions and
notation of Notation \ref{notn: 1}. Then 
$$\mu\left(F_i(a/b)\right)<4^{Hd^2}k^{5d^2}.$$ 
\end{thm}

\begin{proof} By Propositions \ref{prop: lower} and \ref{prop: upper}, we may apply Lemma \ref{ARLemSyn} using
$$1+\varrho=(1-\rho)dk^{m(d-1)} \qquad \vartheta = k, \qquad \delta = 1/(3d^3),\qquad \ell=m(d-1).$$
This gives $$\mu(F_i(a/b))\leq 3(1-\rho)d^4 k^{2m(d-1)}.$$

Since $m$ is the smallest natural number such that $k^m\ge 2^{Hd}k^{2d+3}$, we have $k^m<2^{Hd}k^{2d+4}$.  Thus
\begin{equation}\label{better1}\mu(F_i(a/b))\leq 3(1-\rho)d^4 2^{2Hd(d-1)}k^{2(d-1)(2d+4)}.\end{equation}
Since $d^4\le 2^{2Hd}$, we see that
\begin{align*}\mu(F_i(a/b))&\leq 2^{2Hd}\cdot 3(1-\rho)2^{2Hd(d-1)}k^{2(d-1)(2d+4)}\\ 
&=3(1-\rho)2^{2Hd^2}k^{4d^2+4d-8}\\
&<4^{Hd^2}k^{5d^2},\end{align*} which is the desired result.
\end{proof}

In the proof of our main theorem, we will use the well-known fact that if $\xi$ is a real number, then 
\begin{equation}
\label{eq: wellknownfact}
\mu\left(\frac{x}{y}\cdot\xi\right)=\mu\left(\frac{x}{y}+\xi\right)=\mu(\xi)\end{equation} for any nonzero rational $x/y\in\mathbb{Q}$. 

\begin{prop}\label{mainMinimal} Suppose that $F(x)\in \mathbb{Z}[[x]]$ satisfies $$p(x)+\sum_{i=0}^d a_i(x)F(x^{k^i})=0$$ for some integers $k\geq 2$ and $d\geq 1$ and polynomials $a_0(x),\ldots,a_d(x),p(x)\in\mathbb{C}[x]$ with $a_0(x)a_d(x)\neq 0$, which is minimal with respect to $d$.  Suppose that $a$ and $b$ are integers with $b>0$ satisfying
\begin{enumerate}
\item[(i)] $a/b$ is inside the radius of convergence of $F(x)$,
\item[(ii)] $|a|=b^{\rho}$ for some $\rho\in [0,1/(d+2))$, and
\item[(iii)] $a_0((a/b)^{k^n}) \neq 0$ for all $n\ge 0$.
\end{enumerate}
Then 
$$\mu(F(a/b))\le 4^{H(d+1)^2}k^{5(d+1)^2},$$ where $$H=\max_{0\leq i\leq d}\{\deg p(x), \deg a_i(x)+k^i\nu(a_0(x))\}.$$
If, in addition, we have $a_0(0)\neq 0$, then $H=\max_{0\leq i\leq d}\{\deg p(x), \deg a_i(x)\}$ suffices. 
\end{prop}

\begin{proof} In this proof, we will consider two preliminary cases based on conditions of our coefficient polynomials and then proceed to prove the general result, as the third case, based on the two previous cases.

\emph{Case I.} Assume in addition to the assumptions of the theorem that  $a_0(0)\neq 0$.

Let $F_1(x)=1$ and $F_i(x)=F(x^{k^{i-2}})$ for $i=2,\ldots,d+1$. It is clear that $F_1(x),F_2(x),\ldots,$ $F_{d+1}(x)$ are linearly independent over $\mathbb{Q}(x)$ since $d$ is  minimal. 
Let $\mathbf{A}_F(x)$ denote the $(d+1)\times (d+1)$ matrix
\begin{equation}\label{FAF}\left(\begin{matrix}1 & 0 & 0 & \cdots & 0 \\ 
-\frac{p(x)}{a_0(x)} & -\frac{a_1(x)}{a_0(x)} & -\frac{a_2(x)}{a_0(x)} & \cdots &-\frac{a_d(x)}{a_0(x)}\\
0 &  &  &   &0\\
\vdots &  & I_{d-1} &   &\vdots\\
0 &  &  &  &0 \end{matrix}\right),\end{equation} where $I_{d-1}$ is the $(d-1)\times (d-1)$ identity matrix. 
Then we have 
\begin{equation}
[F_1(x),\ldots ,F_{d+1}(x)]^{\rm T}=\mathbf{A}_F(x)[F_1(x^k),\ldots ,F_{d+1}(x^k)]^{\rm T}.
\end{equation}
Taking the determinant of ${\mathbf A}_F(x)$ by using cofactor expansion along the right-most column gives $$\det \mathbf{A}_F(x)=\frac{a_d(x)}{a_0(x)}\neq 0.$$ Applying Theorem \ref{maineffective} (replacing $d$ by $d+1$),  
gives the result in this case where $$H:=\max_{0\leq i\leq d}\{\deg p(x), \deg a_i(x)\}.$$ 

{\em Case II.} Set $j:=\min_{1\leq i\leq d}\{i: a_i(x)\neq 0\}$ and assume in addition to the assumptions of the theorem that $$\left(k^j-1\right)\cdot \nu({F(x)})\geq \nu({a_0(x)})\quad\mbox{and}\quad \nu(p(x))\geq \nu(a_0(x))+\nu(F(x)).$$

Define the polynomial $A(x)$ and the power series $G(x)$ by $$a_0(x)=x^{\nu({a_0(x)})} A(x)\qquad\mbox{and}\qquad F(x)=x^{\nu(F(x))} G(x).$$ Then \begin{align*} 0 &=p(x)+ a_0(x)F(x)+\sum_{i=1}^d a_i(x)F(x^{k^i})\\ &= x^{-\nu(F(x))-\nu({a_0(x)})}p(x)+A(x)G(x)+\sum_{i=1}^d a_i(x)x^{(k^i-1)\nu(F(x))-\nu({a_0(x)})}G(x^{k^i}),\end{align*} where $x^{-\nu(F(x))-\nu({a_0(x)})}p(x)\in\mathbb{Z}[x]$ since by assumption $\nu(p(x))\geq \nu({a_0(x)})+\nu(F(x)).$ Also, since we have $(k^j-1)\nu(F(x))-\nu({a_0(x)})\geq 0$ and $A(0)\neq 0$, the above equality shows that $G(x)$ satisfies a Mahler-type functional equation with $A(0)a_d(x)x^{(k^d-1)\nu(F(x))-\nu({a_0(x)})}\neq 0$. Thus Case I applies to give that $$\mu\left(G(a/b)\right)\leq 4^{H(d+1)^2}k^{5(d+1)^2},$$ where $$H:=\max_{0\leq i\leq d}\left\{\deg p(x), \deg a_i(x)+\left(k^i-1\right)\cdot \nu({F(x)})\right\}.$$ Since $F(x)=x^{\nu(F(x))} G(x)$ we have that $F(a/b)$ is just a rational multiple of $G(a/b)$, and so $\mu(F(a/b))=\mu(G(a/b))$. This proves the theorem under this added assumption.

{\em Case III.} Assume that neither Case I nor Case II applies. 

We have that $\nu(a_0(x))\neq 0$. Define the series $G(x)$ and the polynomial $P(x)$ by $$G(x):=x^{\nu(a_0(x))}+\sum_{n>\nu(a_0(x))} f(n)x^n\quad\mbox{and}\quad P(x):=F(x)-G(x).$$ Then \begin{align}\nonumber 0 &= p(x)+\sum_{i=0}^d a_i(x)F(x^{k^i})\\
\nonumber &=p(x)+\sum_{i=0}^d a_i(x)P(x^{k^i})+\sum_{i=0}^d a_i(x)G(x^{k^i})\\
\label{PgGg}&=Q(x)+\sum_{i=0}^d a_i(x)G(x^{k^i}),
\end{align} where $Q(x):=p(x)+\sum_{i=0}^d a_i(x)P(x^{k^i}).$ Notice that by construction $\nu(G(x))=\nu(a_0(x)).$ Also, note that by \eqref{PgGg} we have that \begin{equation}\label{mini}\nu(Q(x))=\nu\left(\sum_{i=0}^d a_i(x)G(x^{k^i})\right)\geq \min_{0\leq i\leq d}\left\{\nu(a_i(x))+k^i \nu(G(x))\right\}.\end{equation} 

If the minimum occurring in the right-hand side of \eqref{mini} is at $i=0$, then we can apply Case II to the functional equation \eqref{PgGg} to give \begin{equation}\label{muGg}\mu\left(G(a/b)\right)\leq 4^{H(d+1)^2}k^{5(d+1)^2},\end{equation} where $$H:=\max_{0\leq i\leq d}\left\{\deg Q(x), \deg a_i(x)+\left(k^i-1\right)\cdot \nu({a_0(x)})\right\}.$$ 

If the minimum occurring in the right-hand side of \eqref{mini} is at some $i\geq 1$, then we at least have that $$\nu(Q(x))\geq k\cdot\nu(G(x))\geq 2\cdot\nu(G(x))=\nu(G(x))+\nu(a_0(x)),$$ since $\nu(a_0(x))=\nu(G(x)),$  and again we can apply Case II to give the same bound as in \eqref{muGg}. Since $G(a/b)-F(a/b)$ is rational we have that $\mu(G(a/b))=\mu(F(a/b)).$ Noting that $$\deg Q(x)\leq \max_{0\leq i\leq d}\left\{\deg p(x),\deg a_j(x)+k^i\nu(a_0(x))\right\},$$ gives the desired result, finishing the proof of the theorem.
\end{proof}

Notice that Proposition \ref{mainMinimal} is very similar to Theorem \ref{main}; it has the added condition that $d$ be minimal. In fact, there is no need for this minimality condition, which we will demonstrate in the following results.

Let $k\geq 2$ be a positive integer and let $\Delta_k:\mathbb{Q}[[x]]\to \mathbb{Q}[[x]]$ be defined by $\Delta_k(F(x))=F(x^k).$ Let $F(x)$ be a $k$-Mahler function and $S(\Delta_k)\in\mathbb{Q}(x)[\Delta_k]$ be the minimal degree annihilator of $F(x)$; note that since $\mathbb{Q}(x)[\Delta_k]$ is a left principal ideal domain, $S(\Delta_k)$ is well defined. Write $$S(\Delta_k)=q_0(x)+q_1(x)\Delta_k+\cdots+q_s(x)\Delta_k^s$$ where $q_s(x)q_0(x)\neq 0$ and without loss of generality suppose that $q_i(x)\in\mathbb{Q}[x]$ for $i=0,\ldots, s$ and $\gcd(q_0(x),\ldots,q_s(x))=1$.

Suppose that there is some $$R(\Delta_k)=\frac{a_0(x)}{b_0(x)}+\frac{a_1(x)}{b_1(x)}\Delta_k+\cdots+\frac{a_r(x)}{b_r(x)}\Delta_k^r,$$ with $\gcd(a_i(x),b_i(x))=1$ such that $R(\Delta_k)S(\Delta_k)=h_0(x)+h_1(x)\Delta_k+\cdots+h_{d}(x)\Delta_k^{d}\in\mathbb{Q}[x][\Delta_k]$ is an annihilator of $F(x)$ of degree $s+r$ and height \begin{equation}
\label{eq: new height}
H \ := \ \max\{ \deg(h_i(x))\colon 0\le i\le d\},
\end{equation} where, as is clear from  (\ref{eq: new height}), we use height in this context to mean the maximum of the degrees of the polynomial coefficients.

\begin{prop}\label{htS} The height of $S(\Delta_k)$ is at most $H$, where $H$ is as defined in  (\ref{eq: new height}).
\end{prop}

We will use the following two lemmas to prove Proposition \ref{htS}.

\begin{lem}\label{brq0} We have that $b_r$ divides $q_0(x)q_0(x^k)\cdots q_0(x^{k^{r-1}}).$
\end{lem}

\begin{proof} We will first show that $b_\ell$ divides $q_0(x)q_0(x^k)\cdots q_0(x^{k^{\ell}})$ for all $\ell$. To see this note that the constant coefficient of $R(\Delta_k)S(\Delta_k)$ is $a_0(x)q_0(x)/b_0(x)\in\mathbb{Q}[x]$. Since $\gcd(a_0(x),b_0(x))=1$ we have that $b_0(x)$ divides $q_0(x)$. Now suppose that $b_i$ divides $q_0(x)q_0(x^k)\cdots q_0(x^{k^{i}})$ for all $i<\ell$. The coefficient of $\Delta_k^\ell$ in $R(\Delta_k)S(\Delta_k)$ is $$\sum_{j=0}^\ell \frac{a_j(x)q_{\ell-j}(x^{k^j})}{b_j(x)}=h_\ell(x)\in\mathbb{Q}[x].$$ Thus multiplying through by the product of the $b_j(x)$s, with minor rearrangement we have $$h_\ell(x)\prod_{j=0}^\ell b_j(x)-\sum_{j=0}^{\ell-1} a_j(x)q_{\ell-j}(x^{k^j})\frac{b_0(x)\cdots b_\ell(x)}{b_j(x)}=b_0(x)\cdots b_{\ell-1}(x)a_\ell(x)q_0(x^{k^\ell}) ;$$ we note here that each term in the sum is a polynomial. Since $b_\ell(x)$ divides each term on the left-hand side, it must divide the right-hand side. That is, $b_\ell$ divides $b_0(x)\cdots b_{\ell-1}(x)a_\ell(x)q_0(x^{k^\ell})$. Since for $i<\ell$ we assumed that $b_i$ divides $q_0(x^{k^i})$, we have by transitivity of division that $b_\ell$ divides $q_0(x)q_0(x^k)\cdots q_0(x^{k^{\ell}})$. This completes the first step of the proof of the lemma.

Now let $j\in\{0,\ldots,s\}$ and consider $h_{r+s-j}(x)$, and note that since $h_{r+s-j}(x)$ is the coefficient of $\Delta_k^{r+s-j}$, \begin{equation}\label{LRhrs}\frac{a_r(x)}{b_r(x)}q_{s-j}(x^{k^r})=-\sum_{i=0}^{r-1}\frac{a_i(x)}{b_i(x)}q_{r+s-j-i}(x^{k^i})+h_{r+s-j}(x).\end{equation} By the first step of this proof, we have that the right-hand side of \eqref{LRhrs} is a polynomial if we multiply it by $q_0(x)q_0(x^k)\cdots q_0(x^{k^{r-1}})$, and we thus also have this for the left-hand side of \eqref{LRhrs}; that is, $$\frac{a_r(x)q_{s-j}(x^{k^r})q_0(x)q_0(x^k)\cdots q_0(x^{k^{r-1}})}{b_r(x)}\in\mathbb{Q}[x].$$ Since $\gcd(a_r(x),b_r(x))=1$, we thus have that for each $j\in\{0,\ldots,s\}$, $b_r(x)$ divides the $q_{s-j}(x^{k^r})q_0(x)q_0(x^k)\cdots q_0(x^{k^{r-1}})$. Since $\gcd(q_0(x),\ldots,q_s(x))=1$, there exist $u_0(x),\ldots,u_s(x)$ such that $$\sum_{i=0}^s q_i(x)u_i(x)=1,$$ and since $b_r(x)$ divides the $q_{s-j}(x^{k^r})q_0(x)q_0(x^k)\cdots q_0(x^{k^{r-1}})$ for each $j\in\{0,\ldots,s\}$, we have that $b_r(x)$ divides any linear combination of these. In particular, $b_r(x)$ divides \begin{align*} \sum_{j=0}^s u_{s-j}(x^{k^r})&q_{s-j}(x^{k^r})q_0(x)q_0(x^k)\cdots q_0(x^{k^{r-1}})\\ &= q_0(x)q_0(x^k)\cdots q_0(x^{k^{r-1}})\sum_{j=0}^s u_{s-j}(x^{k^r})q_{s-j}(x^{k^r})\\ &=q_0(x)q_0(x^k)\cdots q_0(x^{k^{r-1}}).\end{align*} This completes the proof of the lemma.
\end{proof}

\begin{lem}\label{ekdH} Let $e_i:=\deg a_i(x)-\deg b_i(x)$ for $i\in\{0,\ldots,r\}$, $d_j:=\deg q_j(x)$ for $j\in\{0,\ldots,s\}$, and $H$ be as given in  (\ref{eq: new height}). Then $e_i+k^id_j\leq H$ for all $i$ and $j$.
\end{lem}

\begin{proof} Pick $j$ such that $d_j$ is maximal and such that $d_k < d_j$ for $k<j$. Next pick the smallest $i$ such that $e_i + k^i d_j$ is maximal; that is, $e_t + k^t d_j < e_i +k^i d_j$ for $t<i$. Now consider the coefficient $$h_{i+j}(x)=\sum_{\ell=0}^{i+j}\frac{a_\ell(x)}{b_\ell(x)}q_{i+j-\ell}(x^{k^\ell})$$ in the product $R(\Delta_k)S(\Delta_k)$. Notice that we look at the term when $\ell=i$ in the above sum we get a contribution of $\frac{a_i(x)}{b_i(x)}q_{j}(x^{k^\ell})$, which has a pole at $x=\infty$ of order $e_i+k^id_j>H$.  Since $\deg h_{i+j}(x)\leq H$, we see that $h_{i+j}(x)$ has a pole at $x=\infty$ of order at most $H$ and so there must be some cancellation in the above sum; that is, necessarily there is some other term that has a pole at $x=\infty$ of order $e_i+k^i d_j$ and so $e_\ell + k^\ell d_{i+j-\ell} = e_i + k^i d_j$ for some $\ell \neq i$.  If $\ell>i$, then $$e_\ell + k^\ell d_{i+j-\ell} < e_\ell + k^\ell d_j \leq  e_i+k^id_j = e_\ell + k^\ell d_{i+j-\ell},$$ a contradiction. And if $\ell<i$, then $$e_\ell + k^\ell d_{i+j-\ell} \leq e_\ell + k^\ell d_j < e_i +k^i d_j= e_\ell + k^\ell d_{i+j-\ell},$$ which is again a contradiction.
\end{proof}

\begin{proof}[Proof of Proposition \ref{htS}] Let $j\in\{0,\ldots,s\}$ be such that $d_j$ is maximal. By Lemma~\ref{brq0} we have that $$e_r=\deg a_r(x)-\deg b_r(x)\geq 0-d_0(1+k+\cdots+k^{r-1})\geq -d_j(1+k+\cdots+k^{r-1}).$$ By Lemma \ref{ekdH} we have that $e_r+k^rd_j\leq H$, and so using the preceding inequality we have $$H\geq e_r+k^rd_j\geq d_j(k^r-k^{r-1}-\cdots-k-1)=d_j\left(\frac{k^{r+1}-2k^r+1}{k-1}\right)\geq d_j,$$ where the last inequality follows from the fact that $k\geq 2$.
\end{proof}

\begin{proof}[Proof of Theorem \ref{main}] Let $d\geq 1$ be an integer and suppose that $F(x)\in \mathbb{Z}[[x]]$ satisfies a Mahler equation $$\sum_{i=0}^d a_i(x)F(x^{k^i})=0,$$ where $a_0(x)a_d(x)\neq 0$. By Proposition \ref{htS}, the Mahler equation which is minimal with respect to length, say \begin{equation}\label{dmin} \sum_{i=0}^{d_{\min}} b_i(x)F(x^{k^i})=0,\end{equation} where $b_0(x)b_{d_{\min}}(x)\neq 0$, has coefficients which are polynomials with integer coefficients which satisfy $\max_{0\leq i\leq d_{\min}}\{\deg b_i(x)\}\leq \max_{0\leq j\leq d}\{\deg a_j(x)\}$.

Now suppose that $a$ and $b$ are integers with $b>1$ satisfying
\begin{enumerate}
\item[(i)] $a/b$ is inside the radius of convergence of $F(x)$,
\item[(ii)] $|a|=b^{\rho}$ for some $\rho\in [0,1/(d+2))$, and
\item[(iii)] $a_0((a/b)^{k^n}) \neq 0$ for $n\ge 0$.
\end{enumerate}
Then by Proposition \ref{mainMinimal} we have 
$$\mu(F(a/b))\le 4^{N(d_{\min}+1)^2}k^{5(d_{\min}+1)^2},$$ where $$N=\max_{0\leq i\leq d_{\min}}\{\deg b_i(x)+k^i\nu(b_0(x))\}.$$ Since $\nu(b_0(x))\leq \deg b_0(x)$, applying Proposition \ref{htS} again, we have $$N=\max_{0\leq i\leq d_{\min}}\{\deg b_i(x)+k^i\nu(b_0(x))\}\leq (k^{d_{\min}}+1)\cdot\max_{0\leq j\leq d}\{\deg a_i(x)\}.$$ Since $d_{\min}\leq d$, the result follows.
\end{proof}

\section{Removal of Mahler's condition}\label{MahlerCond}

While no general method for removing Mahler's condition is known, our interest is focused on irrationality exponents of Mahler numbers, and in this context we are able to remove Mahler's condition. We require Dumas' characterisation of Mahler functions \cite[Theorem 31]{D1993}.

\begin{thm}[Structure Theorem of Dumas \cite{D1993}]\label{Dumas} A $k$-Mahler function is the quotient of a power series and an infinite product which are analytic in the unit disk. Specifically, if ${\bf F}(x):=[F_1(x)=F(x),\ldots,F_d(x)]^T$ satisfies $${\bf F}(x)=\frac{{\bf A}(x)}{B(x)}{\bf F}(x^k),$$ where ${\bf A}(x)\in \mathbb{Q}[x]^{d\times d}$ with $\det{\bf A}(x)\neq 0$ and $B(x)\in\mathbb{Z}[x]$, then there exist $k$-regular series $G_i(x)$ ($i=1,\ldots,d$) such that $$F_i(x)=\frac{G_i(x)}{\prod_{j\geq 0}\beta(x^{k^j})},$$ where $B(x)=\alpha x^{\delta}\beta(x)$, with $\alpha\neq 0$ and $\beta(0)=1$.
\end{thm}

We also require the following lemma, which is based upon \cite[Lemma 6.1]{AB2013}, although more care is needed in keeping track of certain estimates.

\begin{lem} Suppose that $F(x)=\sum_{n\geq 0} f(n)x^n \in \mathbb{Z}[[x]]$ satisfies the Mahler equation \eqref{MFE} and suppose that $F(x)$ is not a polynomial.  If $\delta$ is the order of the zero of $a_0(x)$ at $x=0$ and $H$ denotes the maximum of the degrees of $a_0(x),\ldots ,a_d(x)$, then there exist 
some natural number $M\le H+k^d \delta$ and some polynomial $P(x)$ of degree at most $\delta$ such that $F(x)=P(x)+x^M E(x)$ where $E(0)\neq 0$ and $E(x)$ satisfies a Mahler equation of the form
$$\sum_{i=0}^{d+1} b_i(x) E(x^{k^i}) = 0$$ with $b_0(0)\neq 0$ and such that the degrees of $b_0,\ldots ,b_{d+1}$ are at most $(H+k^d)^2(k+1)$.   Furthermore, the following hold:
\begin{enumerate}
\item[(i)] if $\alpha$ is a nonzero complex number and $a_0(x)$ has a zero at $x=\alpha$ of order $s$, then $b_0(x)$ has a zero at $x=\alpha$ of order at most $s+H+k^d\delta$;
\item[(ii)] if $h$ is the maximum of the absolute values of the coefficients of $a_0(x),\ldots ,a_d(x)$,
then the leading coefficient of $b_0(x)$ is, in absolute value, bounded above by $(d+1)h^2(\delta+1)\cdot\max\{|f(0)|,\ldots,|f(\delta)|\}$.
\end{enumerate}
\label{lem: AB}
\end{lem}

\begin{proof} Let $P(x)$ be the unique polynomial of degree at most $\delta$ such that $F(x)-P(x)$ has a zero at $x=0$ of order at least $\delta+1$.  (In other words, $P(x)$ is the Taylor polynomial approximation to $F(x)$ of degree at most $\delta$.)  Let $M$ denote the order of zero at $x=0$ of $F(x)-P(x)$ and let $E(x)$ be such that $F(x)=P(x)+x^M E(x)$.  Then $E(0)\neq 0$, by construction and $M\geq \delta +1$.  We claim that $M\le H+k^d\delta$.  To see this, let 
\begin{equation}
\label{eq: GQQQ}
Q(x) = \sum_{i=0}^d a_i(x) P(x^{k^i}).
\end{equation}
Then making the substitution $F(x)=P(x)+x^M E(x)$ into the Mahler equation \eqref{MFE} gives
\begin{equation}
\label{eq: GQ}
\sum_{i=0}^d a_i(x) x^{M k^i} E(x^{k^i}) = - Q(x).
\end{equation}
Notice that $Q(x)$ has degree at most $H+k^d \delta$.  Thus if $M>H+k^d\delta$ then we must have $Q(x)=0$ since the left-hand side has a zero of order at least $M$ at $x=0$.  On the other hand if $Q(x)=0$ then we must have $$ \sum_{i=0}^d a_i(x) x^{M k^i} E(x^{k^i}) = 0,$$ and this cannot occur 
if $M>H+k^d\delta$ since the order of vanishing at $x=0$ for 
$a_i(x) x^{M k^i} E(x^{k^i})$ is exactly $\delta+M$ for $i=0$ and is at least $Mk^i>M+\delta$ for $i>0$.  It follows that $M\le H+k^d\delta$.

Observe that the order of vanishing of $a_0(x) x^{M}$ at $x=0$ is precisely $x^{\delta+M}$, which is strictly less than the order of vanishing of $a_i(x) x^{k^i M}$ at $x=0$ for $i>0$ since $M>\delta$.  We let $c_i(x)=a_i(x) x^{k^iM-M-\delta}$ for $i=0,\ldots ,d$ and $R(x)=-Q(x)x^{-M-\delta}$.  Then the $c_i(x)$ and $R(x)$ are all polynomials.  Moreover, each $c_i(x)$ has degree at most $H+k^d M-M-\delta\le H+k^dM $ and $R(x)$ has degree at most $H+k^d\delta-M-\delta \le H+k^d M$.  Then we have
\begin{equation}
\label{eq: GQnew}
\sum_{i=0}^d c_i(x)  E(x^{k^i}) = R(x)
\end{equation}
where $c_0(0)\neq 0$ and $c_i(0)=0$ for $i>0$.  In particular, $R(0)\neq 0$ since $c_0(0)\neq 0$ and $E(0)\neq 0$.  
Applying the operator $x\mapsto x^k$ to  (\ref{eq: GQnew}) gives
\begin{equation}
\label{eq: GQnewk}
\sum_{i=0}^d c_i(x^k)  E(x^{k^{i+1}}) = R(x^k).
\end{equation}
Finally, multiplying both sides of  (\ref{eq: GQnew}) by $R(x^k)$ and subtracting the result of multiplying both sides of  (\ref{eq: GQnewk}) by $R(x)$ gives
\begin{equation}
\label{eq: GQnewk1}
\sum_{i=0}^{d+1} b_i(x)  E(x^{k^{i}}) = 0,
\end{equation}
where $b_i(x)=c_i(x)R(x^k)-c_{i-1}(x^k)R(x),$ and we take $c_{-1}(x)=c_{d+1}(x)=0$. 
We observe now that each $b_i(x)$ has degree at most $(H+k^dM)(k+1)$.  Using the fact that $M\le H+k^d$ now gives us the desired bound of $(H+k^d)^2(k+1)$.  

Finally, we note that $b_0(x)=c_0(x)R(x^k)=-a_0(x)Q(x^k)x^{-M-\delta}$.  Thus the order of vanishing of $b_0(x)$ at $x=\alpha$ is $s$ more than the order of vanishing of $Q(x)$ at $x=\alpha^k$, which is at most $s+H+k^d\delta$, since $Q(x)$ has degree at most $H+k^d\delta$.  Next note that the leading coefficient of $b_0(x)$ is, up to sign, equal to the leading coefficient of $a_0(x)$ times the leading coefficient of $Q(x)$, which is at most $h$ times the absolute value of the leading coefficient of $Q(x)$.
To estimate the size of the leading coefficient of $Q(x)$, we use Equation (\ref{eq: GQQQ})
and observe that $P(x)$ is just $\sum_{j\le \delta} f(j)x^j$. If $Q(x)$ has degree $r$, then we have that the leading coefficient of $Q(x)$ is at most the sum of the absolute values of the coefficients of $x^r$ in $a_i(x)P(x^{k^i})$ as $i$ ranges from $0$ to $d$.  Since $P(x^{k^i})$ has at most $\delta+1$ nonzero coefficients, we see that the coefficient of $x^r$ in $a_i(x) P(x^{k^i})$ is at most $h(\delta+1)$ times the maximum of $|f(0)|,\ldots ,|f(\delta)|$.  Since there are a total of $d+1$ such terms that contribute to the leading coefficient of $Q(x)$, we see that its leading coefficient is at most $(d+1)h(\delta+1)$ times the maximum of $|f(0)|,\ldots ,|f(\delta)|$.  The result now follows.  
\end{proof}

We are now ready to prove the main result of this section.

\begin{thm}\label{thestuff} Suppose that $F(x)=\sum_{n\geq 0} f(n)x^n \in\mathbb{Z}[[x]]$ satisfies the Mahler equation \eqref{MFE} and that $a/b$ is in the radius of convergence of $F(x)$.  Then $$\mu(F(a/b)) \leq (4dHh^3)^{1536 H^6 k^{10d}d^2} k^{320H^4k^{6d}d^2}$$
whenever $\log|a|/\log b\in [0,1/8H^2dk^{2d+1})$,
where $H=\max\{1,\deg a_0(x),\ldots ,\deg a_d(x)\}$ and $h$ is the maximum of $\max\{|f(0)|,\ldots ,|f(H)|\}$ and the absolute values of the coefficients of $a_0(x),\ldots ,a_d(x)$.\end{thm}

\begin{proof} To see how Theorem \ref{thestuff} is deduced from Theorem \ref{maineffective} suppose that $F(x)$ satisfies Mahler's equation \eqref{MFE}, that is, 
$$\sum_{i=0}^d a_i(x)F(x^{k^i})=0,$$ $a/b$ is in the radius of convergence of $F(x)$, and that $\log|a|/\log b\in [0,1/6H^2dk^{2d+1})$. Let $\delta$ be the order of the zero of $a_0(x)$ at $x=0$.

By Lemma \ref{lem: AB}, there exists some natural number $M\le H+k^d \delta$ and some polynomial $P(x)$ of degree at most $\delta\le H$ such that $F(x)=P(x)+x^M E(x)$ where $E(0)=0$ and $E(x)$ satisfies a Mahler equation
$$\sum_{i=0}^{d+1} b_i(x) E(x^{k^i}) \ = \ 0$$ with $b_0(0)\neq 0$ and such that the degrees of $b_0,\ldots ,b_{d+1}$ are at most $(H+k^d)^2(k+1)$.  
It is immediate from Equation (\ref{eq: wellknownfact}) that $\mu(F(a/b))=\mu(E(a/b)).$ We continue by bounding $\mu(E(a/b)).$

To set us up to use Theorem \ref{maineffective}, note that $${\bf E}(x):=[E_1(x)=E(x),\ldots,E_{d+1}(x)]^T$$ satisfies $${\bf E}(x)=\frac{{\bf A}(x)}{B(x)}{\bf E}(x^k),$$ where  $E_j(x)=E(x^{k^{j-1}})$ for $j=1,\ldots,d+1$, $B(x)=b_0(x)$, and $${\bf A}(x)=\left(\begin{matrix}
-b_1(x) & -b_2(x) & \cdots& &-b_{d+1}(x)\\
1 & 0 & \cdots &&0  \\ 
0 & 1 &  &   \\
\vdots &  & \ddots& &   \vdots\\
0 &  &  & 1 &0 \end{matrix}\right).$$

We note that if there does not exist some $i\ge 0$ such that $B((a/b)^{k^i})=0$, then we can apply Theorem \ref{main} to give an upper bound on the irrationality exponent of $E(a/b)$ that is better than the one given in the statement of the theorem.   Thus it is sufficient to consider the case that $B((a/b)^{k^i})=0$ for some $i\ge 0$.  We shall give an upper bound on the irrationality exponent for $E(a/b)$, if $B((a/b)^{k^i})=0$ for some $i\geq 0$, by applying Dumas' theorem and then using L'H\^opital's rule.

By Theorem \ref{Dumas}, there exist $k$-regular functions $G_1(x),\ldots ,G_{d+1}(x)$ such that
$$E_i(x) \ = \ \frac{G_i(x)}{\prod_{j\ge 0} \beta(x^{k^j})},$$ 
where $\beta(x)=b_0(x)/b_0(0)$.

If we let ${\bf G}(x):=[G_1(x)=G(x),\ldots,G_{d+1}(x)]^T$, then we have $${\bf G}(x)=\frac{\beta(x)}{B(x)}{\bf A}(x){\bf G}(x^k)=\frac{{\bf A}(x)}{b_0(0)}{\bf G}(x^k).$$  Differentiating we have that $$\left[\begin{matrix} G_1(x)\\ \vdots\\ G_{d+1}(x)\\ G_1'(x)\\ \vdots\\ G_{d+1}'(x)\end{matrix}\right]=\frac{1}{b_0(0)}\left[\begin{matrix}\mathbf{A}(x) & \mathbf{0}_{r\times r}\\ \mathbf{A}'(x) & kx^{k-1}\mathbf{A}(x)\end{matrix}\right]\left[\begin{matrix} G_1(x^k)\\ \vdots\\ G_{d+1}(x^k)\\ G_1'(x^k)\\ \vdots\\ G_{d+1}'(x^k)\end{matrix}\right].$$ We note that the entries of the above matrix all have degree at most $k-1+(H+k^d)^2(k+1)$. Continuing in this manner one can write for any $n\in\mathbb{N}$ the equation \begin{equation}\label{Gs}{\bf G}_n(x)=\frac{1}{b_0(0)}\left[\begin{matrix}
\mathbf{A}(x) & & &\\ 
* & kx^{k-1}\mathbf{A}(x) &  \\
\vdots & \ddots& \ddots & &  \\
*&\cdots& *& k^nx^{n(k-1)}{\bf A}(x)  
\end{matrix}\right]{\bf G}_n(x^k),\end{equation} where $${\bf G}_n(x):=[G_1(x),\ldots,G_{d+1}(x),G_1'(x),\ldots,G_{d+1}'(x),\ldots,G_1^{(n)}(x),\ldots,G_{d+1}^{(n)}(x)]^T,$$ and the upper triangular part of the matrix in the equality is all zeros.  Again the degrees of all entries in this matrix are bounded by
$n(k-1)+(H+k^d)^2(k+1)$.

We are now set up to show that Mahler's condition can be removed. Towards this, suppose that $a,b\in\mathbb{Z}$, $b>0$, and $\gcd(a,b)=1$.

Since $B(x)=b_0(x)$, Lemma \ref{lem: AB} gives that the leading coefficient of $B(x)$ is at most
$$(d+1)(\delta+1)h_1^2\max\{|f(0)|,\ldots ,|f(\delta)|\},$$
where $h_1$ is the maximum of the absolute values of the coefficients of $a_0(x),\ldots ,a_d(x)$.  We note that $h_1\le h$ and $\max\{|f(0)|,\ldots ,|f(\delta)|\}\le \max\{|f(0)|,\ldots ,|f(H)|\}\le h$ and so the leading coefficient of $B(x)$ is at most $(d+1)(H+1)h^3$.   Using the rational roots theorem and the fact that $\gcd(a,b)=1$, we have that 
$B((a/b)^{k^n})\neq 0$ whenever $b^{k^n}>(d+1)(H+1)h^3$.  In particular, there exists a unique positive integer $N$ such that
\begin{equation}
k^{N-1} \le \max\left\{1,\frac{\log((d+1)(H+1)h^3)}{\log b}\right\} < k^{N}.
\label{eq: kN}
\end{equation}
Then, by the remarks above, we have that for all $n\ge N$, $$B\left(\left(\frac{a}{b}\right)^{k^n}\right)\neq 0.$$ 

If $E(a/b)$ is defined, then $G(a/b)$, as defined above, must be zero. Define the integers $s_i\geq 0$ and the polynomial $\tau(x)$ by $$\beta(x)=\tau(x)\prod_{i=0}^{N-1}\Big(x-\left(\frac{a}{b}\right)^{k^i}\Big)^{s_i},$$ where $\tau((a/b)^{k^i})\neq 0$ for $i\geq 0$. Thus there is a polynomial $T(x)\in\mathbb{Q}[x]$ such that $$\prod_{i=0}^{N-1}\beta(x^{k^i})=\left(x-\frac{a}{b}\right)^sT(x),$$ where $s=s_0+\cdots+s_{N-1}$ and $T(a/b)\neq 0$; note that $0\leq s\leq \deg \beta(x)$. Then $G(x)$ has a zero of order at least $s$ at $x=a/b$. A routine calculation shows that $$E\left(\frac{a}{b}\right)=\frac{G^{(s)}\left(\frac{a}{b}\right)}{s!\cdot T\left(\frac{a}{b}\right)\prod_{i\geq N}\beta\left(\left(\frac{a}{b}\right)^{k^i}\right)}.$$ Now $s!\cdot T(a/b)\in\mathbb{Q}\setminus\{0\},$ so using Equation (\ref{eq: wellknownfact}), we see that we need not worry about it in order to determine the irrationality exponent.

Set $$L_{i,j}(x)=\frac{G_i^{(j)}(x)}{\prod_{n\geq N}\beta(x^{k^n})},$$ and write $${\bf L}(x):=\frac{{\bf G}_s(x)}{\prod_{n\geq N}\beta(x^{k^n})},$$ so that we have   \begin{equation}\label{Li}{\bf L}(x)=\frac{1}{b_0(0)\cdot\beta(x^{k^N})}\left[\begin{matrix}
\mathbf{A}(x) & & &\\ 
* & kx^{k-1}\mathbf{A}(x) &  \\
\vdots & \ddots& \ddots & &  \\
*&\cdots& *& k^sx^{s(k-1)}{\bf A}(x)  
\end{matrix}\right]{\bf L}(x^k).\end{equation}

Recall that $\beta((a/b)^{k^N+i})\neq 0$ for all $i\geq 0$ and that $\deg\beta(x)=\deg B(x)$. Also, we have that $$E\left(\frac{a}{b}\right)\cdot s!\cdot T\left(\frac{a}{b}\right)=\frac{G^{(s)}\left(\frac{a}{b}\right)}{\prod_{i\geq N}\beta\left(\left(\frac{a}{b}\right)^{k^i}\right)}=L_{1,s}\left(\frac{a}{b}\right).$$  Finally, we note that the matrix in  (\ref{Li}) has entries whose degrees are at most 
$$s(k-1)+(H+k^d)^2(k+1)$$ and $\beta(x^{k^N})$ has degree at most
$$k^N (H+k^d)^2(k+1),$$ which by  (\ref{eq: kN}) gives
\begin{equation}
\deg\beta(x^{k^N}) \le \max\left( 1, \frac{\log((d+1)(H+1)h^3)}{\log b}\right) (H+k^d)^2(k+1)k.
\label{eq: beta}
\end{equation}
Thus using the vector equation of length $(s+1)(d+1)$ in \eqref{Li}, we can apply Theorem \ref{maineffective} to $L_{1,s}(a/b)$\footnote{We must actually work with a vector equation of length $(s+1)(d+1)+1$ to apply Theorem \ref{maineffective} since we must throw the constant function $1$ into our system.  This has the net effect of adding a new row and column to the matrix in \eqref{Li} that has their shared entry exactly equal to $b_0(0)\beta(x^{k^N})$ and all other entries equal to zero.}
to obtain 
\begin{equation}
\label{eq: muFfirst}
\mu(F(a/b))=\mu(E(a/b))=\mu(L_{1,s}(a/b)) \leq 4^{H_0{(1+(s+1)(d+1))}^2}k^{5{(1+(s+1)(d+1))}^2},
\end{equation}
 where 
\begin{equation}
\label{eq: HH0}
H_0=\max\left(\max\left( 1, \frac{\log((d+1)(H+1)h^3)}{\log b}\right)(H+k^d)^2(k+1)k, s(k-1)+(H+k^d)^2(k+1)\right).
\end{equation}
The bound $s\leq \deg \beta(x)$ implies $$s+1\le 1+\deg \beta(x)\leq 1+\deg B(x)\leq 1+(H+k^d)^2(k+1).$$

Thus $(s+1)(d+1)+1\le (1+(H+k^d)^2(k+1))(d+1)+1$.  It is straightforward to check that
$(1+(H+k^d)^2(k+1))(d+1)+1\le 8H^2k^{2d+1}d$ for $d,H\ge 1$ and $k\ge 2$.  
Using this estimate along with (\ref{eq: muFfirst}), we see that
$$\mu(F(a/b))\leq 4^{H_0(8H^2k^{2d+1}d)^2}k^{5{(8H^2k^{2d+1}d)}^2}.$$
To finish the proof, we estimate $H_0$.  
Notice from the above remarks, we have $$s(k-1)+(H+k^d)^2(k+1)\le (H+k^d)^2(k+1)(k-1)+(H+k^d)^2(k+1)=(H+k^d)^2k(k+1)$$ and it is straightforward to check that this is at most
$8H^2k^{2d+2}$ for $H,d\ge 1$ and $k\ge 2$.  
On the other hand, for $H,d\ge 1$ and $k\ge 2$, we have
$$\frac{\log((d+1)(H+1)h^3)}{\log b}(H+k^d)^2(k+1)k\le 12\log((d+1)(H+1)h^3)H^2 k^{2d+2}.$$
Since $(H+k^d)^2 (k+1)k \le 12\log((d+1)(H+1)h^3)H^2 k^{2d+2}$,
the definition (\ref{eq: HH0}) of $H_0$ gives    
$$H_0\le 12\log((d+1)(H+1)h^3)H^2 k^{2d+2}.$$
Using this estimate for $H_0$ now gives
$$\mu(F(a/b))\leq 4^{768\log((d+1)(H+1)h^3)H^6 k^{6d+4}d^2}k^{5{(8H^2k^{2d+1}d)}^2}.$$
Since $4<e^2$, $6d+4\le 10d$, and $4d+2\le 6d$, we see that this estimate is less than 
$$((d+1)(H+1)h^3)^{1536 H^6 k^{10d}d^2} k^{320H^4k^{6d}d^2}.$$ We now obtain the desired result by noting that $(d+1)(H+1)h^3\le 4dHh^3$.
Moreover, we notice this upper bound 
applies whenever
$$\rho=\log|a|/\log(b)\le 1/((s+1)(d+1)+1),$$
thus, in particular, it applies for $\rho < 1/8H^2dk^{2d+1}$. 
\end{proof}

\section{Free modules over polynomial rings}\label{QS}

In this section, we prove a quantitative result on the unimodular completion of matrices with coefficients in $\mathbb{Q}[x]$ which is used in the next section to give a quantitative bound on the irrationality exponent of regular numbers (see Proposition~\ref{QSprop} and Theorem \ref{regularkkernel}). This result is a quantitative version of a special case of the following result of Quillen and Suslin: given a unimodular $m\times n$ matrix ${\bf A}$ ($m\leq n$) over $\mathbb{C}[x_1,\ldots,x_r]$, there exists a unimodular $n\times n$ matrix ${\bf U}$ over $\mathbb{C}[x_1,\ldots,x_r]$ such that $${\bf A}\cdot{\bf U}=\left[\begin{matrix} 1&0&\cdots&0&0&\cdots&0\\ *&1&\cdots&0&0&\cdots&0\\ \vdots&\ddots&\ddots&\vdots&\vdots&&\vdots\\ *&\cdots&*&1&0&\cdots&0\end{matrix}\right].$$ See Fitchas and Galligo \cite{FG1990} and Logar and Sturmfels \cite{LS1992} for more general algorithms and less specialised degree bounds on the entries of ${\bf U}$.

\begin{defn} Let $A$ be an $m\times n$ matrix with $m\le n$ and with entries in a commutative ring $R$.  We say that $A$ is \emph{unimodular} if the induced map $A:R^n\to R^m$ is surjective.
\end{defn}

\begin{defn} Let $A$ and $B$ be matrices of the same size with entries in a commutative ring $R$. We say that matrix $A$ is {\em row equivalent} to a matrix $B$ if matrix $B$ can be obtained by applying a finite sequence of invertible elementary row operations to $A$. We will denote this equivalence by $A\equiv B$.
\end{defn}

\begin{lem}\label{get0s} Let $\mathbf{A}(x)\in\mathbb{Q}[x]^{m\times n}$ be a matrix with polynomial entries of degree at most $H$. Then there is a polynomial $a(x)\in\mathbb{Q}[x]$ with $\deg a(x)\leq H$ and a matrix $\mathbf{B}(x)\in\mathbb{Q}[x]^{m\times (n-1)}$ with polynomial entries of degree at most $2H$ such that $\mathbf{A}(x)$ is row equivalent to the $m\times n$ matrix \begin{equation}\label{AB}\left[\begin{matrix} a(x) &  
\\ \mathbf{0}_{(m-1)\times 1} &\mathbf{B}(x) & \end{matrix}
\right].\end{equation}
\end{lem}

\begin{proof} Let $\mathbf{A}(x)=(a_{i,j}(x))\in\mathbb{Q}[x]^{m\times n}$ be a matrix with polynomial entries of degree at most $H$. If the first column of $\mathbf{A}(x)$ contains only zero entries then the result follows trivially. Thus we can assume that there is some nonzero entry in the first column. Interchanging rows as necessary, we can assume that the $(1,1)$-entry of $\mathbf{A}(x)$ is nonzero of minimal degree within the first column. 

By now adding the appropriate multiple of row $1$ to each of the other rows we can reduce the degrees of the first entry of each row $i$ for $i\in\{2,\ldots,m\}$, and we have that \begin{equation}\label{A1}\mathbf{A}(x)\equiv \left[\begin{matrix} a_{1,1}^{(1)}(x) & & & \\  \vdots & & \mathbf{B}^{(1)}(x) & \\ a_{m,1}^{(1)}(x) & & & \end{matrix}\right],\end{equation} where $\deg a_{i,1}^{(1)}(x)<\deg a_{1,1}^{(1)}(x)\leq H$ for $i=2,\ldots,m$, and the entries of $\mathbf{B}^{(1)}(x)$ are bounded in degree by $2H-\deg a_{1,1}^{(1)}(x)=2H-\deg a_{1,1}^{(1)}(x).$ Now interchange rows so that again the $(1,1)$-entry of the matrix on the right-hand side of \eqref{A1} is of minimal degree within the first column (note that this degree is now strictly less than $H$). Repeat the process of adding the appropriate multiple of the new row $1$ to each of the other rows to again reduce the degrees of the first entry of each row $i$ for $i\in\{2,\ldots,m\}$. Thus we have that \begin{equation}\label{A2}\mathbf{A}(x)\equiv \left[\begin{matrix} a_{1,1}^{(2)}(x) & & & \\  \vdots & & \mathbf{B}^{(2)}(x) & \\ a_{m,1}^{(2)}(x) & & & \end{matrix}\right],\end{equation} where $$\deg a_{i,1}^{(2)}(x)<\deg a_{1,1}^{(2)}(x)<\deg a_{1,1}^{(1)}(x)\leq H$$ for $i=2,\ldots,m$, and the entries of $\mathbf{B}^{(2)}(x)$ are bounded in degree by $$H+H-\deg a_{1,1}^{(1)}(x)+\deg a_{1,1}^{(1)}(x)-\deg a_{1,1}^{(2)}(x)=2H-\deg a_{1,1}^{(2)}(x).$$ Continuing in this manner, since the maximum degree of the elements of the first column decreases at each step, there is a $k$ with $0\leq k\leq H$, such that \begin{equation}\label{Ak}\mathbf{A}(x)\equiv \left[\begin{matrix} a_{1,1}^{(k)}(x) & & & \\ \mathbf{0}_{(n-1)\times 1} & &\mathbf{B}^{(k)}(x) & \end{matrix}\right],\end{equation} where $\deg a_{1,1}^{(k)}(x)\leq H$ and the entries of $\mathbf{B}^{(k)}(x)$ are bounded in degree by $$2H-\deg a_{1,1}^{(k)}(x)\leq 2H.$$ This proves the lemma.
\end{proof}

\begin{rem} Note that $a(x)$ in the matrix in \eqref{AB} as given by the proof of Lemma~\ref{get0s} is the greatest common divisor of the entries in the first column of $\mathbf{A}(x)$.
\end{rem}

\begin{lem}\label{2mH} If a $\mathbb{Q}[x]$-submodule $W$ of $\mathbb{Q}[x]^d$ is spanned by polynomial vectors whose coordinates all have degree at most $H$, then $W$ has a $\mathbb{Q}(x)$-basis consisting of polynomial vectors whose coordinates all have degree at most $2^{m} H$, where $m=\rank(W)$. 
\end{lem}

\begin{proof} For $i\in\{1,\ldots,n\}$ let $\mathbf{v}_i(x)$ be the $1\times d$ vector whose coordinates all have degree at most $H$ such that $W=\spn_{1\leq i\leq n}\{\mathbf{v}_i(x)\}$. Applying Lemma \ref{get0s} we have that $$\left[\begin{matrix} \mathbf{v}_1(x)\\ \vdots\\ \mathbf{v}_n(x)\end{matrix}\right]\equiv \left[\begin{matrix} a_1(x)& *& &&&&  &\cdots&*\\ 0&\cdots&0& a_2(x) &*&\cdots&&&*\\ &&&&\ddots&&&\\ &&&&0&a_m(x)&*&\cdots &*\\ &&&&&0&&&\\ \vdots&&&&&\vdots&&\mathbf{B}(x)&\\ 0&&&&\cdots&0&&&\\
\end{matrix}\right],$$ where the entries in row $i$ are bounded in degree by $2^{i}H$ for $i\in\{1,\ldots,m\}$ and $\mathbf{B}(x)$ is a matrix whose entries are of degree at most $2^mH$. Now if there are any nonzero entries in $\mathbf{B}(x)$ then we have $\rank(W)>m$, and so since this is not the case, we necessarily have that $\mathbf{B}(x)$ has only zero entries. Thus the first $m$ row vectors of the right-hand matrix in the above equivalence are a $\mathbb{Q}[x]$-module basis for $W$. Since these vectors have entries of degree at most $2^{m}H$ we arrive at the desired result.
\end{proof}

\begin{lem} Let $d\geq 2$ and $\mathbf{w}(x)\in\mathbb{Q}[x]^d$ be a unimodular row all of whose coordinates have degree bounded by $H$. Then there exist $d$ polynomials $q_1(x),\ldots,q_d(x)\in\mathbb{Q}[x]$ each of degree less than $(d-1)H$ such that $$[q_1(x),\ldots,q_d(x)]\cdot \mathbf{w}=1.$$
\end{lem}

\begin{proof} Write $\mathbf{w}(x)=[w_1(x), w_2(x),\ldots ,w_d(x)].$  We may assume without loss of generality that the degree of $w_d$ is at least as large as the degree of $w_i$ for $i<d$.  By unimodularity, there exist polynomials $q_1(x),\ldots, q_d(x)$ such that 
$$\sum_i q_i(x)w_i(x)=1.$$  We pick $q_1,\ldots ,q_d$ satisfying this equation with 
$$\max({\rm deg}(q_1),\ldots ,{\rm deg}(q_d)\}$$ minimal.  We claim that each $q_i$ has degree less than $(d-1)H$.  To see this, let $t_i(x)=\prod_{j\neq i} w_j(x)$ for $i=1,\ldots ,d$.  Then each $t_i(x)$ has degree at most $H(d-1)$ and so by the division algorithm, we may write 
$q_i(x) = s_i(x)t_i(x)+r_i(x)$ for some polynomials $s_1(x),\ldots ,s_d(x)$.  We note that $r_i(x)$ is nonzero for some $i<d$ since otherwise, we would have $w_d(x)|q_i(x)$ for $i<d$, which is impossible.

By construction, $t_i(x)w_i(x)- t_d(x)w_d(x)=0$ for all $i$ and so we see
$$[r_1(x),r_2(x),\ldots ,r_{d-1}(x), s_1(x)t_d(x)+\cdots +s_{d-1}(x)t_d(x) + q_d(x)] \mathbf{w}(x)=1.$$
Let $u(x)=s_1(x)t_d(x)+\cdots +s_{d-1}(x)t_d(x) + q_d(x)$.  Then
$$u(x)w_d(x) = 1 -\sum_{i<d} r_i(x)w_i(x).$$  Since the right-hand side has degree strictly less than $(d-1)H+{\rm deg}(w_d)$, we see that $u(x)$ has degree strictly less than $(d-1)H$, which gives the desired result.
\end{proof}

\begin{lem}\label{kerAq} Let $\mathbf{w}(x)\in\mathbb{Q}[x]^d$ be a unimodular row with entries that have degree bounded by $H$ and let $A:\mathbb{Q}[x]^d\to\mathbb{Q}[x]$ be the map defined by $A(\mathbf{v}(x))=\mathbf{v}(x)\cdot \mathbf{w}(x).$ Then $\ker A$ is spanned by a finite number of vectors all of whose coordinates have degree bounded by $H$. In particular, $\ker A$ has a $\mathbb{Q}[x]$-module basis consisting of vectors whose coordinates have degree at most $2^{d-1}H.$
\end{lem}

\begin{proof} Let $d\geq 2$ and $\mathbf{w}(x)\in\mathbb{Q}[x]^d$ be a unimodular row with entries that have degree bounded by $H$ and let $A:\mathbb{Q}[x]^d\to\mathbb{Q}[x]$ be the map defined by $A(\mathbf{v}(x))=\mathbf{v}(x)\cdot \mathbf{w}(x).$ For $i,j\in\{1,\ldots,d\}$, let $\mathbf{b}_{i,j}(x)$ denote the $1\times d$ row vector with $-w_j(x)$ in the $i$-th position and $w_i(x)$ in the $j$-th position 
and with all other entries equal to zero.  We note that $$\ker A\supseteq \spn_{1\leq i<j\leq d}\left\{\mathbf{b}_{i,j}(x)\right\}.$$  We let $W_0$ denote the $\mathbb{Q}[x]$-span of $\{\mathbf{b}_{i,j}(x)\}$.  We claim that $\ker A$ is spanned by $W_0$ along with a set of vectors with the property that every coordinate has degree strictly less than $H$.

To see this, note that we may assume without loss of generality that the degree of $w_d(x)$ is at least as big as the degree of $w_i(x)$ for $i<d$.  Let $\mathbf{u}(x)\in \ker A$ and let $u_i(x)$ denote the $i$-th coordinate of $\mathbf{u}(x)$.  Then, for each $i<d$, we have $\deg(u_i(x)-p_i(x)w_d(x))<\deg(w_d(x))$ for some polynomial $p_i(x)$.  Then
$\mathbf{u}_0(x):=\mathbf{u}(x)+\sum_{i<d} p_i(x)\mathbf{b}_{i,d}(x)$ has the property that each of its coordinates has degree less than $H$, except for possibly the $d$-th coordinate.  We let $u_{0,i}(x)$ denote the $i$-th coordinate of $\mathbf{u}_0(x)$ for $i\in \{1,\ldots ,d\}$.  Then we have
$$\sum_i u_{0,i}(x) w_i(x) = 0.$$  By construction,
$$\sum_{i<d} u_{0,i}(x) w_i(x)$$ has degree at most $2\deg(w_d(x))-1$.  It follows that $u_{0,d}(x)$ must have degree at most $\deg(w_d(x))-1$ and so $\mathbf{u}_0(x)$ has the property that each of its coordinates has degree at most $H-1$.  Thus we see that we may use the procedure above to produce a basis for $W_0$ consisting of the elements $\left\{\mathbf{b}_{i,j}(x)\right\}$ along with a set of vectors with the property that every coordinate has degree strictly less than $H$.  Thus $\ker A$ is spanned by a set of vectors with the property that every coordinate in each vector has degree at most $H$.  Thus there is a $\mathbb{Q}[x]$-module basis for $\ker A$ consisting of polynomial vectors whose coordinates all have degree at most $2^{d-1}H$ by Lemma~\ref{2mH}.
\end{proof}

\begin{lem} Let $\mathbf{w}(x)\in\mathbb{Q}[x]^d$ be a unimodular row with entries that have degree bounded by $H$ and let $A:\mathbb{Q}[x]^d\to\mathbb{Q}[x]$ be the map defined by $A(\mathbf{v}(x))=\mathbf{v}(x)\cdot \mathbf{w}(x).$ Suppose that $\mathbf{q}(x)\in\mathbb{Q}[x]^d$ is a unimodular row with entries that have degree bounded by $(d-1)H$ such that $A(\mathbf{q}(x))=1$. Then $\mathbf{q}(x)$ taken together with a $\mathbb{Q}[x]$-module basis for $\ker A$ is a $\mathbb{Q}[x]$-module basis for $\mathbb{Q}[x]^d$. Moreover, this basis can be taken so that the coordinates of the basis vectors have degree bounded by $2^{d-1}H$.
\end{lem}

\begin{proof} This follows from $$0\rightarrow{} \ker A\xrightarrow{{\rm id}} \mathbb{Q}[x]^{d-1}\oplus \mathbb{Q}[x]\xrightarrow{A} \mathbb{Q}[x]\xrightarrow{} 0.$$ Since this exact sequence splits, we have that $$\ker A \oplus \mathbb{Q}[x]\cong \mathbb{Q}[x]^d.$$ Now since $\mathbf{q}(x)$ is a unimodular row such that $A(\mathbf{q}(x))=1$, the above isomorphism implies that $\mathbf{q}(x)$ taken together with a $\mathbb{Q}[x]$-module basis for $\ker A$ is a $\mathbb{Q}[x]$-module basis for $\mathbb{Q}[x]^d$. By Lemma \ref{kerAq} since $\ker A$ has a basis with coordinates of degree at most $2^{d-1}H$ and since $\mathbf{q}(x)$ has coordinates of degree at most $(d-1)H$ this basis for $\mathbb{Q}[x]^d$ can be taken so that the coordinates of the basis vectors have degree bounded by $2^{d-1}H$. This proves the lemma.
\end{proof}

We are now in a position to prove the main goal of this section (see also Logar and Sturmfels \cite{LS1992} or Fitchas and Galligo \cite{FG1990}). This result is a direct consequence of the following lemma.

\begin{lem}\label{QSlem} Let $W$ be a $\mathbb{Q}[x]$-submodule of $\mathbb{Q}[x]^d$ of rank $m$ that has $\mathbb{Q}[x]$-module basis of row vectors $\mathbf{w}_1(x),\ldots,\mathbf{w}_m(x)$ whose coordinates all have degree at most $H$. Then there is a $\mathbf{G}(x)\in {\rm SL}_d(\mathbb{Q}[x])$ with entries of degree bounded by $(m+1)H$ such that $$\left[\begin{matrix} \mathbf{w}_1(x)\\ \mathbf{w}_2(x)\\ \vdots\\ \mathbf{w}_m(x)\end{matrix}\right]\mathbf{G}(x)=\left[\begin{matrix} a(x) & 0 & \cdots & 0\\ c_1(x)&&&\\ \vdots&&\mathbf{B}(x)&\\ c_{m-1}(x)&&&\end{matrix}\right]$$ where $\mathbf{B}(x)$ is an $(m-1)\times(d-1)$ matrix, and all of the entries of $\mathbf{B}(x)$, as well as $a(x),c_1(x),\ldots,$ $c_{m-1}(x)$, have degree at most $2H$. Moreover, the matrix $(\mathbf{G}(x))^{-1}$ has entries of degree bounded by $(2m-1)H$.
\end{lem}

\begin{proof} To prove this lemma, we need only apply Lemma \ref{get0s} to $$\left[\begin{matrix} \mathbf{w}_1(x)\\ \mathbf{w}_2(x)\\ \vdots\\ \mathbf{w}_m(x)\end{matrix}\right]^T$$ noting that, in the proof of Lemma \ref{get0s}, we produce matrices $\mathbf{G}(x),\mathbf{B}(x)\in{\rm SL}_d(\mathbb{Q}[x])$ such that $$\mathbf{G}(x)^T\left[\begin{matrix} \mathbf{w}_1(x)\\ \mathbf{w}_2(x)\\ \vdots\\ \mathbf{w}_m(x)\end{matrix}\right]^T=\left[\begin{matrix} a(x) & 0 & \cdots & 0\\ c_1(x)&&&\\ \vdots&&\mathbf{B}(x)&\\ c_{m-1}(x)&&&\end{matrix}\right]^T,$$ where $a(x)$ is the greatest common divisor of the entries of $\mathbf{w}_1(x)$, and the entries of $\mathbf{B}(x)$ all have degree at most $2H$.  

There is necessarily an $m\times m$ invertible submatrix, $\mathbf{W}(x)$, consisting of some subset of $m$ columns of  $$\left[\begin{matrix} \mathbf{w}_1(x)\\ \mathbf{w}_2(x)\\ \vdots\\ \mathbf{w}_m(x)\end{matrix}\right].$$ We let $\mathbf{D}(x)$ denote the $m\times m$ submatrix of 
$$\left[\begin{matrix} a(x) & 0 & \cdots & 0\\ c_1(x)&&&\\ \vdots&&\mathbf{B}(x)&\\ c_{m-1}(x)&&&\end{matrix}\right]^T$$ made from the corresponding columns.  Then we have
$$\mathbf{W}(x)\mathbf{G}(x)=\mathbf{D}(x),$$ and so $\mathbf{D}(x)$ is invertible and we have
$$\mathbf{G}(x)=\mathbf{W}(x)^{-1}\mathbf{D}(x)$$ and
$$\mathbf{G}(x)^{-1} = \mathbf{D}(x)^{-1}\mathbf{W}(x).$$
We see that
${\rm det}(\mathbf{W}(x))\cdot \mathbf{G}(x) = \adj(\mathbf{W}(x))\cdot\mathbf{D}(x)$, where, as previously, $\adj(\mathbf{E})$ is the classical adjoint of a matrix $\mathbf{E}$.  Since the entries of $\adj(\mathbf{W}(x))$ all have degree at most $(m-1)H$ and the entries of $\mathbf{D}(x)$ all have degree at most $2H$, we see that
${\rm det}(\mathbf{W}(x))\cdot \mathbf{G}(x)$ is a polynomial matrix whose entries have degree at most $(m+1)H$.   Since $\mathbf{G}(x)$ is a polynomial matrix and ${\rm det}(\mathbf{W}(x))$ is a polynomial, we see that the entries of $\mathbf{G}(x)$ are each of degree at most $(m+1)H$.  A similar argument shows that the entries of $\mathbf{G}(x)^{-1}$ are each of degree at most $(m-1)2H+H=(2m-1)H$.   
\end{proof}

\begin{prop}\label{QSprop} Let $W$ be a $\mathbb{Q}[x]$-submodule of $\mathbb{Q}[x]^d$ of rank $m$ that has $\mathbb{Q}[x]$-module basis of row vectors $\mathbf{w}_1(x),\ldots,\mathbf{w}_m(x)$ whose coordinates all have degree at most $H$. Then there exist $m$ matrices $\mathbf{B}_1(x),\mathbf{B}_2(x),\ldots,\mathbf{B}_m(x)\in{\rm SL}_d(\mathbb{Q}[x])$ such that $$\left[\begin{matrix} \mathbf{w}_1(x)\\ \mathbf{w}_2(x)\\ \vdots\\ \mathbf{w}_m(x)\end{matrix}\right]\mathbf{B}_1(x)\mathbf{B}_2(x)\cdots\mathbf{B}_m(x)=\left[\begin{matrix} a_1(x) &0&\cdots&0&0&\cdots&0\\ *&a_2(x) &\cdots&0&0&\cdots&0\\ \vdots&\ddots&\ddots&\vdots&\vdots&&\vdots\\ *&\cdots&*&a_m(x) &0&\cdots&0\end{matrix}\right]$$ is lower triangular with nonzero polynomials on the diagonal. Moreover, the entries of the matrices $\mathbf{B}_1(x)\mathbf{B}_2(x)\cdots\mathbf{B}_m(x)$ and $(\mathbf{B}_1(x)\mathbf{B}_2(x)\cdots\mathbf{B}_m(x))^{-1}$ are bounded in degree by $(m+1)H2^m$ and $(2m+1)H2^m$ respectively.  In addition, if the matrix $$\left[\begin{matrix} \mathbf{w}_1(x)\\ \mathbf{w}_2(x)\\ \vdots\\ \mathbf{w}_m(x)\end{matrix}\right]$$ is unimodular then we can take $a_1(x)=\cdots =a_m(x)=1$.  
\end{prop}

\begin{proof} The bounds on degrees of entries follow immediately by induction using Lemma \ref{QSlem}.  If $$\left[\begin{matrix} \mathbf{w}_1(x)\\ \mathbf{w}_2(x)\\ \vdots\\ \mathbf{w}_m(x)\end{matrix}\right]$$ is unimodular, then so is $$\left[\begin{matrix} a_1(x) &0&\cdots&0&0&\cdots&0\\ *&a_2(x) &\cdots&0&0&\cdots&0\\ \vdots&\ddots&\ddots&\vdots&\vdots&&\vdots\\ *&\cdots&*&a_m(x) &0&\cdots&0\end{matrix}\right].$$  This can only occur if $a_1(x),\ldots ,a_d(x)$ are units in $\mathbb{Q}[x]$, and thus by right-multiplying by an invertible scalar matrix if necessary, we may assume that $a_1(x)=\cdots =a_m(x)=1$. 
\end{proof}

\section{Application to $k$-regular power series}\label{muauto}

Concerning the irrationality exponent of an automatic number, Adamczewski and Cassaigne \cite{AC2006} proved the following result.  We recall for the reader's benefit that the definition of the $k$-kernel of a $k$-automatic sequence is given in the introduction.

\begin{thm}[Adamczewski and Cassaigne \cite{AC2006}]\label{ACthm2} Let $k,b\geq 2$ be two integers and let ${\bf a}=\{a(n)\}_{n\geq 0}$ be a $k$-automatic sequence taking values in $\{0,1,\ldots,b-1\}$. Let $m$ be the cardinality of the $k$-kernel of ${\bf a}$ and let $d$ be the number of values in $\{0,1,\ldots,b-1\}$ which the sequence ${\bf a}$ actually assumes. Then $$\mu\left(\sum_{n\geq 0}\frac{a(n)}{b^n}\right)\leq dk(k^m+1).$$
\end{thm}

\noindent In this section, as an application of Theorem \ref{maineffective}, we provide a generalisation of Theorem \ref{ACthm2}. In particular, we prove the following result.

\begin{thm}\label{regularkkernel} Let $k\geq 2$ be an integer, $\mathbf{f}:=\{f(n)\}_{n\geq 0}$ be a $k$-regular sequence, and $F(x)=\sum_{n\geq 0}f(n)x^n\in\mathbb{Z}[[x]]$. Let $L$ denote the dimension of the $\mathbb{Q}$-vector space spanned by the $k$-kernel of $\mathbf{f}$ and $a/b$ be a rational number with $\log|a|/\log b\in[0,1/(L+2))$ and $b\geq 2$. Then $$\mu\left(F(a/b)\right)\leq 39^{(2^{6}k)^L}.$$
\end{thm}

To this end, let $k\geq 2$ be a positive integer and $f:\mathbb{N}\cup\{0\}\to\mathbb{Z}\subseteq \mathbb{Q}$ be a $k$-regular sequence. Let \begin{equation}\label{f1L}\{f(n)\}_{n\geq 0}=\{f_1(n)\}_{n\geq 0},\{f_2(n)\}_{n\geq 0},\ldots,\{f_L(n)\}_{n\geq 0}\end{equation} be a basis for the $\mathbb{Q}$-vector space spanned by $\mathcal{K}_k(\{f(n)\}_{n\geq 0})$. Let \begin{equation}\label{F1L}F_i(x):=\sum_{n\geq 0} f_i(n)x^n\in\mathbb{Z}[[x]]\end{equation} for $i=1,\ldots,L$. Then $$F_i(x)=\sum_{j=0}^{k-1}\sum_{n\geq 0} f_i(kn+j)(x^k)^n x^j,$$ and by construction, each $\{f_i(kn+j)\}_{n\geq 0}$ is a $\mathbb{Q}$-linear combination of $\{f_1(n)\}_{n\geq 0},$ $\ldots,$ $\{f_L(n)\}_{n\geq 0}.$ Thus we may write \begin{equation}\label{FpF}F_i(x)=\sum_{j=1}^{L}p_{i,j}(x)F_j(x^k)\end{equation} where for each $i,j$ we have $p_{i,j}(x)\in\mathbb{Q}[x]$ and $\deg p_{i,j}(x)\leq k-1$. Thus, writing $\mathbf{F}(x):=[F_1(x),\ldots,F_L(x)]^T,$ we have \begin{equation}\label{Areg}\mathbf{F}(x)=\mathbf{A}(x)\mathbf{F}(x^k),\end{equation} where $$\mathbf{A}(x)=(p_{i,j}(x))_{1\leq i,j\leq L}\in \mathbb{Q}[x]^{L\times L},$$ and the polynomials $p_{i,j}(x)$ are given by \eqref{FpF}.

Let $S$ denote the dimension of the span of $\{F_1(x),\ldots,F_L(x)\}$ as a $\mathbb{Q}(x)$-vector space; that is, \begin{equation}\label{S} S:=\dim_{\mathbb{Q}(x)}\sum_{i=1}^L \mathbb{Q}(x) F_i(x),\end{equation} and note that $1\leq S\leq L$. Then \begin{equation}\label{Wker}W:=\left\{[\psi_1(x),\ldots,\psi_L(x)]\in\mathbb{Q}(x)^L: \sum_{i=1}^L \psi_i(x)F_i(x)=0\right\}\end{equation} is a subspace of $\mathbb{Q}(x)^L$ of codimension $S$. We can pick a spanning subset $A$ of $W$ with the following properties: 
\begin{enumerate}
\item[(i)] $A\subseteq\mathbb{Q}[x]^L$, and
\vspace{.1cm}
\item[(ii)] if $D$ is the largest degree polynomial which occurs as a coordinate of some $\mathbf{w}(x)\in A$, then any other spanning subset of $W$ satisfying (i) has some polynomial of degree at least $D$ occurring as a coordinate; that is, $A$ is a minimal spanning set of $W$ with respect to maximum degree of coordinates of elements of $A$.
\end{enumerate}

We have the following lemma.

\begin{lem}\label{Dequal1} Let $k\ge 2$ be a positive integer, $F_1(x),\ldots ,F_L(x)\in \mathbb{Q}[[x]]$ be as defined in \eqref{F1L}, and $W$ be as given in \eqref{Wker}. Then there is a spanning set for $W$ each of whose coordinates are polynomials of degree at most $2^L(k^L-1)$.
\end{lem}

\begin{proof} Note that for the system in question, ${\bf A}(x)$ is an $L \times L$ polynomial matrix and has entries of degree at most $k-1$. Set $${\bf B}(x):={\bf A}(x){\bf A}(x^k)\cdots {\bf A}(x^{k^{L-1}})$$ and denote by $K_{{\bf B}(x)}$ the left kernel of ${\bf B}(x)$. Note that the entries of ${\bf B}(x)$ are all of degree at most $k^L-1$ and observe that $K_{{\bf B}(x)}$ is equal to the left kernel of $${\bf A}(x){\bf A}(x^k)\cdots {\bf A}(x^{k^{j}})$$ for all $j>L-1$ as well. 

We start our proof, by finding a basis for $K_{{\bf B}(x)}$, each of whose coordinates are polynomials of degree at most $2^{L}(k^L-1)$. We use a reduction similar to that within the proof of Lemma \ref{2mH}; we row reduce $\left[{\bf B}(x)|{\bf I}_{L\times L}\right]$ using Lemma \ref{get0s} to obtain the row equivalent matrix $$\left[{\bf B}(x)|{\bf I}_{L\times L}\right]\equiv \left[\left.\begin{matrix} a_1(x)& *& &&&&  &\cdots&*\\ 0&\cdots&0& a_2(x) &*&\cdots&&&*\\ &&&&\ddots&&&\\ &&&&0&a_m(x)&*&\cdots &*\\ &&&&&0&&\cdots&0\\ \vdots&&&&&\vdots&\ddots&&\vdots\\ 0&&&&\cdots&0&&\cdots&0\\
\end{matrix}\right|{\bf C}(x)\right],$$ where $m=\rank{\ {\bf B}(x)}$, and the matrix ${\bf C}(x)$ has entries whose coordinates are polynomials of degree at most $$2^{L}(k^L-1).$$ 

Performing row operations on the matrix
$\left[{\bf B}(x)|{\bf I}_{L\times L}\right]$ is equivalent to left multiplying by an invertible $L\times L$ matrix; in this case the row operations are the result of left multiplying by ${\bf C}(x)$.   Thus the bottom $L-m$ rows of ${\bf C}(x){\bf B}(x)$ are zero.  Since ${\bf C}(x)$ is invertible and ${\bf B}(x)$ has rank $m$, we see that the bottom $L-m$ rows of ${\bf C}(x)$ form a basis, which we denote by $\mathcal{B}_0$, for $K_{{\bf B}(x)}$.

Also, $\mathcal{B}_0\subset W$ since if $[c_1(x),\ldots ,c_L(x)]$ is in $\mathcal{B}_0$, then by construction we have
\[0= [c_1(x),\ldots ,c_L(x)]{\bf B}(x)[F_1(x^{k^L}),\ldots ,F_L(x^{k^L})]^T = \sum_i c_i(x)F_i(x)\]
and so $[c_1(x),\ldots ,c_L(x)]\in W$.

If $\mathcal{B}_0$ spans $W$, then we are done. If not, let $$\{{\bf w}_1(x),{\bf w}_2(x),\ldots,{\bf w}_n(x)\}, \quad \left({\bf w}_i(x)=[w_{i,1}(x),\ldots,w_{i,L}(x)]\right)$$ be a such that $$\mathcal{B}_0\cup \{{\bf w}_1(x),{\bf w}_2(x),\ldots,{\bf w}_n(x)\}$$ is a basis for $W$ with ${\bf w}_i(x)\notin K_{{\bf B}(x)}$ for each $i$, and with $$D:=\max\{\deg w_{i,j}(x)\}$$ minimal according to property (ii) above. Note that to prove the lemma, it is sufficient to prove that $D=1$.

To this end, set $${\bf v}_i(x)={\bf w}_i(x){\bf B}(x).$$ Then writing, ${\bf v}_i(x)=[v_{i,1}(x),\ldots,v_{i,L}(x)]$, we have that the $v_{i,j}(x)$ are polynomials of degree at most $D+k^L-1,$ and also that $${\bf v}_i(x){\bf F}(x^{k^L})={\bf w}_i(x){\bf B}(x){\bf F}(x^{k^L})={\bf w}_i(x){\bf F}(x)=0.$$ For each $\ell\in\{0,1,\ldots,k^L-1\}$, define ${\bf v}_i^{(\ell)}(x)$ by 
\begin{equation}
\label{eq: vil}
{\bf v}_i(x)=\sum_{\ell=0}^{k^L-1}x^\ell {\bf v}_i^{(\ell)}(x^{k^L}).
\end{equation} By construction ${\bf v}_i^{(\ell)}(x^{k^L}){\bf F}(x^{k^L})=0,$ so that ${\bf v}_i^{(\ell)}(x){\bf F}(x)=0$ for each $\ell\in\{0,1,\ldots,k^L-1\}$; that is, each ${\bf v}_i^{(\ell)}(x)\in W$. Also, at least one of the ${\bf v}_i^{(\ell)}(x)$ is nonzero, since otherwise ${\bf w}_i(x)$ would be in $K_{{\bf B}(x)}$. Define $$W_0:=\spn\left(\mathcal{B}_0\cup\left\{{\bf v}_i^{(\ell)}(x):\ell\in\{0,1,\ldots,k^L-1\},i\in\{1,\ldots, n\}\right\}\right).$$ 

We claim that the set $W_0=W$. To see this, note that if $W_0\neq W$, then the dimension of $W_0/K_{{\bf B}(x)}$ is strictly less than $m$. Since right multiplication by ${\bf B}(x)$ has kernel $K_{{\bf B}(x)}$, applying this map to $W_0$ induces a map on $W_0/K_{{\bf B}(x)}$ and thus $W_0 {\bf B}(x)$ has dimension strictly less than $m$ as well. It follows that there are $L-m+1$ linearly independent column vectors ${\bf t}_1(x),\ldots,{\bf t}_{L-m+1}(x)$ in $\mathbb{Q}(x)^L$ such that $$W_0 {\bf B}(x) {\bf t}_s(x) = 0$$ for all $s$. Thus $${\bf v}_i^{(\ell)}(x^{k^L}){\bf B}(x^{k^L}) {\bf t}_s(x^{k^L}) = 0,$$ for all $i,\ell$ and $s$. This then gives that $${\bf v}_i(x){\bf B}(x^{k^L}) {\bf t}_s(x^{k^L}) = 0,$$ for all $i$ and $s$, which in turn gives that $${\bf w}_i(x){\bf B}(x){\bf B}(x^{k^L}) {\bf t}_s(x^{k^L}) = 0,$$ for all $i$ and $s$. 
Since $\mathcal{B}_0\cup \{{\bf w}_1(x),{\bf w}_2(x),\ldots,{\bf w}_n(x)\}$ is a basis for $W$ and $\mathcal{B}_0$ is a basis for $K_{{\bf B}(x)}$, we have
that the images of ${\bf w}_1(x),{\bf w}_2(x),\ldots,{\bf w}_n(x)$ in $W/K_{{\bf B}(x)}$ form a basis for the quotient space.  In particular, ${\bf w}_1(x){\bf B}(x),\ldots ,{\bf w}_n(x){\bf B}(x)$ are linearly independent.  By construction, the left kernel of ${\bf B}(x)$ is equal to the left kernel of ${\bf B}(x){\bf B}(x^{k^L})$ and so we have that the set $$\left\{{\bf w}_i(x){\bf B}(x){\bf B}(x^{k^L}): i=1,\ldots,n\right\}$$ is linearly independent. Define $${\bf u}_i(x):={\bf w}_i(x){\bf B}(x){\bf B}(x^{k^L}).$$  Then we have $${\bf u}_i(x){\bf t}_s(x^{k^L})=0$$ for all $i$ and $s$. That is, we have an $m$ dimensional set (the span of the ${\bf u}_i(x)$) whose orthogonal complement is at least $L-m+1$ dimensional, a contradiction. Thus $W_0=W$.

We are now in a position to prove that $D=1$. Note that since $F_1(x),\ldots ,F_L(x)$ are linearly independent over $\mathbb{Q}$, we have $D\geq 1$. By the definition of $W_0$, the fact that $W_0=W$, and by the minimality of $D$, we must have that the maximum degree of the coordinates of the ${\bf v}_i^{(\ell)}(x)$ is at least $D$, since otherwise we could pick some subset whose images in $W/K_{{\bf B}(x)}$ form a basis for the quotient space; this would then contradict the minimality of $D$.

Using Equation (\ref{eq: vil}), we see $$\deg {\bf v}_i^{(\ell)}(x)\leq \frac{D+k^L-1-\ell}{k^L},$$ and so we have that $$D\leq \max_{0\leq \ell\leq k^{L-1}}\left\{\frac{D+k^L-1-\ell}{k^L}\right\}=1-\frac{D-1}{k^L}.$$ If $D>1$, then we must have $L=0$, which we are assuming is not the case. Thus $D=1$. Hence there is a spanning set for $W$, namely $$\mathcal{B}_0\cup\left\{{\bf w}_i(x)\colon i\in\{1,\ldots, n\}\right\},$$ whose elements have polynomial coordinates of degree at most $2^L(k^L-1)$.
\end{proof}

\begin{lem}\label{unibasis} Let $S$ be as given in \eqref{S}, $W$ be as given in \eqref{Wker}, and suppose that $\{\mathbf{w}_1(x),\ldots,\mathbf{w}_{L-S}(x)\}$ is a $\mathbb{Q}[x]$-module basis for $W\cap \mathbb{Q}[x]^L$. Then the matrix ${\bf T}(x)$ whose $i$-th row is $\mathbf{w}_{i}(x)$ is unimodular.  
\end{lem}

\begin{proof} We have a surjective $\mathbb{Q}[x]$-module homomorphism $\varphi:\mathbb{Q}[x]^L \to U:=\sum_{i=1}^L \mathbb{Q}[x]F_i(x),$ with kernel $W\cap \mathbb{Q}[x]^L$.  Since $U$ is finitely generated and torsion free as a $\mathbb{Q}[x]$-module, the map splits and we have $\mathbb{Q}[x]^L = W\oplus U.$  Moreover, if we choose $\mathbf{v}_1(x),\ldots,\mathbf{v}_S(x) \in \mathbb{Q}[x]^L$ such that $\varphi(\mathbf{v}_1(x)),\ldots,\varphi(\mathbf{v}_S(x))$ is a basis for $U$ then $\mathbf{w}_1(x),\ldots,\mathbf{w}_{L-S}(x),\mathbf{v}_1(x),\ldots,\mathbf{v}_S(x)$ is a basis for $\mathbb{Q}[x]^L$. It follows that the $L \times L$ matrix $${\bf V}(x):=[\mathbf{w}_1(x),\ldots,\mathbf{w}_{L-S}(x),\mathbf{v}_1(x),\ldots,\mathbf{v}_S(x)]^T$$ is invertible and hence has determinant equal to a nonzero scalar in $\mathbb{Q}$.  In particular the column span of ${\bf V}(x)$ is $\mathbb{Q}[x]^L$ and so the column span of ${\bf T}(x)$ is $\mathbb{Q}[x]^{L-S}$, which says that ${\bf T}(x)$ is unimodular.
\end{proof}

Applying Lemma \ref{2mH}, Lemma \ref{Dequal1}, and Lemma \ref{unibasis} immediately gives the following result.

\begin{lem}\label{WbasisS} Let $k\ge 2$ be a positive integer, $F_1(x),\ldots ,F_L(x)\in \mathbb{Q}[[x]]$ be as defined in \eqref{F1L}, and $W$ be as given in \eqref{Wker}. Then $W\cap \mathbb{Q}[x]^L$ has a $\mathbb{Q}[x]$-module basis consisting of polynomial vectors whose coordinates all have degree at most $2^{2L-S}(k^L-1)$, where $S$ is as defined in \eqref{S}.
\end{lem}

We are now ready to prove Theorem \ref{regularkkernel}.

\begin{proof}[Proof of Theorem \ref{regularkkernel}]
Let $\mathbf{w}_1(x),\ldots,\mathbf{w}_{L-S}(x)$ be the $\mathbb{Q}[x]$-module basis for $W\cap \mathbb{Q}[x]^L$ as given by Lemma \ref{WbasisS} and, as above, define $$\mathbf{T}(x):=\left[\begin{matrix}\mathbf{w}_1(x)\\  \vdots\\ \mathbf{w}_{L-S}(x)\end{matrix}\right].$$ By Proposition \ref{QSprop} there is a $\mathbf{B}(x)\in{\rm SL}_L(\mathbb{Q}[x])$, such that $\mathbf{B}(x)$ and $(\mathbf{B}(x))^{-1}$ have entries of degree at most $(L-S+1)2^{3L-2S}(k^L-1)$ and $(2L-2S+1)2^{3L-2S}(k^L-1)$, respectively, and $$\mathbf{T}(x)\mathbf{B}(x)=\left[\begin{matrix} 1&0&\cdots&0&0&\cdots&0\\ *&1&\cdots&0&0&\cdots&0\\ \vdots&\ddots&\ddots&\vdots&\vdots&&\vdots\\ *&\cdots&*&1&0&\cdots&0\end{matrix}\right],$$ since ${\bf T}(x)$ is unimodular.  Define $$\mathbf{M}(x):=\left[\begin{matrix} \mathbf{T}(x)\mathbf{B}(x) \\ \mathbf{0}_{S\times (L-S)}\mid\mathbf{I}_{S\times S}\end{matrix}\right],$$ and note that by construction $\det \mathbf{M}(x)=1$ and the matrix $\mathbf{T}(x)\mathbf{B}(x){\bf M}(x)^{-1}$ is the the first $L-S$ rows of the $L\times L$ identity matrix.  Now set $$\mathbf{Y}(x):=\mathbf{M}(x)\mathbf{B}(x)^{-1}=\left[\begin{matrix} \mathbf{T}(x) \\ \mathbf{R}(x)\end{matrix}\right],$$ for some $S\times L$ matrix $\mathbf{R}(x)$. By construction we have that $\det \mathbf{Y}(x)=1$ and that the entries of $\mathbf{Y}(x)$ are polynomials of degree at most $(2L-2S+1)2^{3L-2S}(k^L-1)$. Moreover, the matrix $(\mathbf{Y}(x))^{-1}=\mathbf{B}(x)$ has entries of degree at most $(L-S+1)2^{3L-2S}(k^L-1)$

By the definition of $\mathbf{Y}(x)$, we have $$\mathbf{Y}(x)\mathbf{F}(x)=\left[\begin{matrix} \mathbf{0}_{(L-S)\times 1}\\ G_{1}(x)\\ \vdots\\ G_S(x)\end{matrix}\right],$$ where each of $G_{1}(x),\ldots,G_S(x)$ is a $\mathbb{Q}[x]$-linear combination of $F_1(x),\ldots,F_L(x)$.  We claim that $G_{1}(x),\ldots,$ $G_S(x)$ are linearly independent over $\mathbb{Q}(x)$.  To see this, suppose that there is some non-trivial linear combination that is equal to zero:
$$\sum_{i=1}^S u_i(x) G_i(x) \ = \ 0.$$  Let $${\bf u}(x)=[\mathbf{0}_{1\times (L-S)}~u_1(x)~\cdots ~u_S(x)].$$
Then by construction, we have $${\bf u}(x){\bf Y}(x){\bf F}(x)=0.$$  In particular, since ${\bf u}(x){\bf Y}(x)$ is a vector that annihilates ${\bf F}(x)$, we have that it is a $\mathbb{Q}(x)$-linear combination of ${\bf w}_1(x),\ldots ,{\bf w}_{L-S}(x)$.  That is, there is some ${\bf v}(x)\in \mathbb{Q}(x)^{1\times (L-S)}$ such that ${\bf v}(x){\bf T}(x)={\bf u}(x){\bf Y}(x)$.  But ${\bf Y}(x)={\bf M}(x){\bf B}(x)^{-1}$ and so we see that
$${\bf v}(x){\bf T}(x){\bf B}(x){\bf M}(x)^{-1}={\bf u}(x).$$  
But this is impossible, since \[ \mathbf{T}(x)\mathbf{B}(x) {\bf M}(x)^{-1}
= \left[{\bf I}_{(L-S)\times (L-S)}~\mathbf{0}_{(L-S)\times S}\right]\]   
and so the last $S$ entries of ${\bf v}(x){\bf T}(x){\bf B}(x){\bf M}(x)^{-1}$ must be zero, contradicting the fact that ${\bf u}(x)$ is nonzero.  

Let $a/b\in\mathbb{Q}$ with $b>1$ and $\log|a|/\log b \in [0,1/(L+2)).$ Note that we are interested in an irrationality measure of $F(a/b)$ and also that $F_1(x)=F(x)$ is in the $\mathbb{Q}[x]$-span of $G_{1}(x),\ldots,G_S(x)$. Thus there is a $\mathbb{Q}$-linear combination $$H(x):=\sum_{j=1}^S r_jG_j(x)$$ of $G_{1}(x),\ldots,G_S(x)$ such that when specialised at $x=a/b$ we have $H(a/b)=F(a/b)$ and for at least one $j_0\in\{1,\ldots,S\}$ we have $r_{j_0}\neq 0$. Permuting $G_{1}(x),\ldots,G_S(x)$ if needed, we can assume that $j_0=S$. Thus we have, setting $$\mathbf{R}:=\left[\begin{matrix}\mathbf{I}_{(L-1)\times L}\\ \mathbf{0}_{1\times (L-S)} \qquad r_1\ \cdots\ r_S\end{matrix}\right]\in\mathbb{Q}^{L\times L},$$ that $\det \mathbf{R}\neq 0$ and $$\mathbf{R}\cdot\mathbf{Y}(x)\mathbf{F}(x)=\left[\begin{matrix} \mathbf{0}_{(L-S)\times 1}\\ G_{1}(x)\\ \vdots\\ G_{S-1}(x)\\ H(x)\end{matrix}\right].$$ 

Since $\mathbf{F}(x)=\mathbf{A}(x)\mathbf{F}(x^k)$ we have that $$(\mathbf{R}\cdot\mathbf{Y}(x))^{-1}\left[\begin{matrix} \mathbf{0}_{(L-S)\times 1}\\ G_1(x)\\ \vdots\\ G_{S-1}(x)\\ H(x)\end{matrix}\right]=\mathbf{A}(x)(\mathbf{R}\cdot\mathbf{Y}(x^k))^{-1}\left[\begin{matrix} \mathbf{0}_{(L-S)\times 1}\\ G_1(x^k)\\ \vdots\\ G_{S-1}(x^k)\\ H(x^k)\end{matrix}\right],$$ and so $$\left[\begin{matrix} \mathbf{0}_{(L-S)\times 1}\\ G_1(x)\\ \vdots\\ G_{S-1}(x)\\ H(x)\end{matrix}\right]=\mathbf{R}\cdot\mathbf{Y}(x)\mathbf{A}(x)(\mathbf{R}\cdot\mathbf{Y}(x^k))^{-1}\left[\begin{matrix} \mathbf{0}_{(L-S)\times 1}\\ G_1(x^k)\\ \vdots\\ G_{S-1}(x^k)\\ H(x^k)\end{matrix}\right].$$ Define the matrices $\mathbf{C}_{1,1}(x),\mathbf{C}_{1,2}(x),\mathbf{C}_{2,1}(x),$ and $\mathbf{C}_{2,2}(x)$ by $$\mathbf{R}\cdot\mathbf{Y}(x)\mathbf{A}(x)(\mathbf{R}\cdot\mathbf{Y}(x^k))^{-1}=\left(\begin{matrix} \mathbf{C}_{1,1}(x)&\mathbf{C}_{1,2}(x)\\ \mathbf{C}_{2,1}(x)& \mathbf{C}_{2,2}(x)\end{matrix} \right),$$ where $\mathbf{C}_{1,1}(x)$ is $(L-S)\times (L-S)$, $\mathbf{C}_{1,2}(x)$ is $(L-S)\times S$, $\mathbf{C}_{2,1}(x)$ is $S\times (L-S)$, and $\mathbf{C}_{2,2}(x)$ is $S\times S$. Then $$\left[\begin{matrix} G_1(x)\\ \vdots\\ G_{S-1}(x)\\ H(x)\end{matrix}\right]=\mathbf{C}_{2,2}(x)\left[\begin{matrix} G_1(x^k)\\ \vdots\\ G_{S-1}(x^k)\\ H(x^k)\end{matrix}\right].$$ Note that $$\mathbf{R}\cdot\mathbf{Y}(x)\mathbf{A}(x)(\mathbf{R}\cdot\mathbf{Y}(x^k))^{-1}\in\mathbb{Q}[x]^{S\times S},$$ so that the matrix $\mathbf{C}_{2,2}(x)$ has the form \begin{equation}\label{CB}\mathbf{C}_{2,2}(x)=\left(\frac{c_{i,j}(x)}{q}\right)_{1\leq i,j\leq S}\end{equation} where $q\in\mathbb{Z}\setminus \{0\}$ and all of the polynomials $c_{i,j}(x)$ are polynomials with integer coefficients and have degree bounded by the degrees of the entries of $$\mathbf{R}\cdot\mathbf{Y}(x)\mathbf{A}(x)(\mathbf{R}\cdot\mathbf{Y}(x^k))^{-1},$$ degrees for which the bound 
\begin{align}
\nonumber L(2L-2S+1)&2^{3L-2S}(k^L-1)+(k-1)\\
\nonumber &=L(2L+1)2^{3L-2S}(k^L-1)+(k-1)-2SL2^{3L-2S}(k^L-1)\\
\nonumber&\leq 2L^2 2^{3L-2}k^L+L2^{3L-2S}(k^L-1)+(k-1)-2L2^{3L-2S}(k^L-1)\\
\nonumber&\leq \frac12 L^2 2^{3L}k^L+(k-1)-L2^{L}(k^L-1)\\
\nonumber&\leq \frac12 L^2 2^{3L}k^L+(k-1)-2(k^L-1)\\
\label{eq: importantbound}& <\frac{1}{2}L^22^{3L}k^L.
\end{align}
holds. 

Now note that since $G_1(x),\ldots,G_{S-1}(x),H(x)$ are linearly independent over $\mathbb{Q}(x)$, the matrix $\mathbf{C}_{2,2}(x)$ is invertible. So also the matrix $$\left[\begin{matrix} 1 & \mathbf{0}_{1\times S}\\ \mathbf{0}_{S\times 1} & \mathbf{C}_{2,2}(x)\end{matrix}\right]$$ is invertible. Thus we can apply Theorem \ref{maineffective} to the system $$ \left[\begin{matrix} 1\\G_1(x)\\ \vdots\\ G_{S-1}(x)\\ H(x)\end{matrix}\right]=\left[\begin{matrix} 1 & \mathbf{0}_{1\times S}\\ \mathbf{0}_{S\times 1} & \mathbf{C}_{2,2}(x)\end{matrix}\right]\left[\begin{matrix} 1\\G_1(x^k)\\ \vdots\\ G_{S-1}(x^k)\\ H(x^k)\end{matrix}\right],$$ using $H$ as given in the right-hand side of \eqref{eq: importantbound} and taking $d=S+1\le L+1$; since $H(a/b)=F(a/b)$, this gives us that the irrationality exponent of $F(a/b)$ is at most $$2^{L^2 (L+1)^2 2^{3L}k^L} k^{5(L+1)^2}.$$

\noindent It is straightforward to check, using elementary estimates, that this is in turn bounded by $39^{(2^{6}k)^L}$, which is the desired result.
\end{proof}

\section{Approximation of regular numbers by algebraic numbers}\label{regAB}

Adamczewski and Bugeaud \cite{AB2007,AdBu11} showed that automatic 
irrational numbers  are transcendental and are not  $U$-numbers
within Mahler's classification recalled in the Introduction.
In this section, we use the system constructed 
at the end of the proof of Theorem \ref{regularkkernel} 
together with Lemma \ref{right} and the $p$-adic Schmidt Subspace 
Theorem to extend their result to irrational regular numbers. 

\begin{thm}\label{thmNotUeffective}
Let $k\geq 2$ be an integer, $\mathbf{f}:=\{f(n)\}_{n\geq 0}$ be a $k$-regular sequence, and $F(x)=\sum_{n\geq 0}f(n)x^n\in\mathbb{Z}[[x]]$. Let $L$ denote the dimension of the $\mathbb{Q}$-vector space spanned by the $k$-kernel of $\mathbf{f}$ and $a/b$ be a rational number with $\rho := \log|a|/\log b\in[0,1/(L+2))$ and $b\geq 2$. Then $F(a/b)$ is either rational or transcendental. Moreover, if $m\geq 2$ is an integer, then we have $$
w_m (F(a/b)) \leq \max\left\{39^{(2^{6}k)^L}, 
(8 m)^{c   (\log 8 m) (\log \log 8 m)} \right\}, 
$$
where the constant $c$ can be explicitly given
and depends on $F(x)$ and $\rho$, but is independent of $b$ and $m$.
\end{thm}

Since, by Theorem \ref{regularkkernel},
we already know that $F(a/b)$ is not a Liouville number, that is, $w_1(F(a/b))$ is finite,
it only remains for us to control the approximation to $F(a/b)$ by algebraic numbers 
of any given degree $m \ge 2$. We will not only show that $F(a/b)$
is transcendental or rational, but also bound from below the distance
between $F(a/b)$, when irrational, and any irrational algebraic number. 
To do this, we proceed as in \cite{AdBuMes,AdBu11} using a suitable
quantitative version of the Schmidt Subspace Theorem. 

Since automatic numbers form a subset of regular numbers, Theorem \ref{thmNotUeffective} implies
that any irrational automatic number is either an $S$-number or a 
$T$-number, a result already established in \cite{AdBu11}. Although both 
proofs depend ultimately on the Schmidt Subspace Theorem, there are some important differences.
Indeed, the
transcendence results obtained in \cite{AB2007,AdBu11} depend on a
result of Cobham \cite{Cob} asserting that an automatic sequence
has sublinear complexity, while in the present paper we make use of the
functional equation satisfied by the generating series of an
automatic sequence.

Before proceeding with the proof of Theorem \ref{thmNotUeffective}, we present
a few auxiliary results. 

We state below an immediate consequence of Theorem 5.1 of \cite{BuEv08}.

For a linear form $L=\sum_{i=1}^n \alpha_i X_i$ with real algebraic coefficients,
let $\mathbb{Q} (L)$ denote the number field generated by $\alpha_1, \ldots , \alpha_n$ and
define the height of $L$ by
$$
H^*(L):=\prod_{i=1}^n \, d H(\alpha_i),
$$ 
where $d$ is the degree of $\mathbb{Q} (L)$.

Let $n$ be an integer with $n\geq 2$. Let $\varepsilon$ be a positive real number
and $\mathcal{S}=\{ \infty ,p_1, \ldots , p_t\}$ be a finite subset of
the set of places of $\mathbb{Q}$ containing the infinite place. Further, let $L_{i, \infty}$ 
$(i=1, \ldots , n)$ be linear forms in $X_1, \ldots ,  X_n$ with real algebraic coefficients and
$L_{i, p}$ $(p\in \{p_1, \ldots ,  p_t\},\,
i=1, \ldots ,  n)$ be linear forms in $X_1, \ldots ,  X_n$   
such that
\begin{align}
\nonumber &|\det (L_{1, \infty}, \ldots ,  L_{n, \infty})| =1, \\
\nonumber & [\mathbb{Q} (L_{i,  \infty}):\mathbb{Q} ]\leq D\ \hbox{for $i=1, \ldots ,  n$,}\\
\nonumber & H^*(L_{i, \infty})\leq H \ \hbox{for $i=1, \ldots ,  n$,} 
\end{align}
and let $e_{i, p}$ $(p\in \mathcal{S},\,i=1, \ldots ,  n)$ be real numbers satisfying $e_{i, \infty}\leq 1$ $(i=1, \ldots ,  n),$ and $e_{i, p}\leq 0$ $(p\in \mathcal{S}\setminus\{\infty\},\, i=1, \ldots ,  n),$ such that 
$$\sum_{p\in \mathcal{S}}\sum_{i=1}^n e_{i, p} = -\varepsilon .$$

Let $\Vert {\bf x} \Vert$ denote the sum of the absolute values of the 
coordinates of the integer vector ${\bf x}$.
We consider
the system of inequalities
\begin{equation}\label{y5.7}
|L_{i, p}({\bf x})|_p\leq \Vert {\bf x} \Vert^{e_{i, p}}\ (p\in \mathcal{S},\, i=1, \ldots ,  n)
\ \hbox{in ${\bf x} \in\mathbb{Z}^n$ with ${\bf x} \not= 0$.}
\end{equation}

\begin{thm}\label{thm5.1}
The set of solutions of \eqref{y5.7} with
$$
\Vert {\bf x} \Vert  >\max \left\{ 2H, n^{2n/\varepsilon}\right\}
$$
is contained in the union of at most
\begin{equation}\label{y5.9}\begin{cases}
8^{(n+9)^2}(1+\varepsilon^{-1})^{n+4}\log (4nD)\log\log (4nD) &
\hbox{if $n\geq 3$}\\
2^{32}(1+\varepsilon^{-1})^3
\log (4nD)\log\big( (1+\varepsilon^{-1})\log (4nD)\big)&
\hbox{if $n=2$}.\end{cases}\end{equation}
proper linear subspaces of $\mathbb{Q}^n$.
\end{thm}

We will make use of the following lemmas.

\begin{lem}\label{lem8.1}
Let $m \ge 1$ be an integer. For every real irrational number $\xi$ we have
$$
\frac{w_m (\xi) + 1}{2} \le w_m^* (\xi) \le w_m(\xi).
$$
In particular, $w_m(\xi)$ is finite if, and only if, $w_m^*(\xi)$ is finite.
\end{lem}

For a proof of Lemma \ref{lem8.1}, see e.g. \cite{BuLiv}.

\begin{lem}\label{lem8.2} 
Let $\alpha$ be a real algebraic number
of degree $m$ and $P(X)$ be an integer polynomial of degree $n$.
If $P(X)$ does not vanish at $\alpha$, then 
$$
|P(\alpha)| \ge (m+1)^{-n-1} \, (n+1)^{-m-1} \, H(\alpha)^{-n} \, H(P)^{-m+1}.
$$
\end{lem}

Lemma \ref{lem8.2} follows from \cite[Theorem A.1]{BuLiv}. 

Let us recall that the height $H({\bf p})$ of an integer vector ${\bf p}$ in $\mathbb{Z}^n$ is the 
maximum of the absolute values of its coefficients 
and that the height $H(\mathcal{H})$ of a hyperplane $\mathcal{H}$ of 
$\mathbb{Q}^n$ given by the equation $y_1 x_1 + \ldots + y_n x_n = 0$, where
$y_1, \ldots , y_n$ are integers, not all zero and without a common prime divisor, 
is equal to the maximum of the absolute values of $y_1, \ldots , y_n$.

For given vectors
${\bf x}_1, \ldots, {\bf x}_n$ in $\mathbb{Z}^n$, 
we denote by
$\rank( {\bf x}_1, \ldots, {\bf x}_n)$ 
the dimension of the $\mathbb{Q}$-vector space generated by these vectors. 

\begin{lem}\label{lem8.4}
Let $n$ and $N$ 
be integers with $N > 2^n$ and
${\bf p}_1,{\bf p}_2, \ldots , {\bf p}_N$ 
be nonzero elements of $\mathbb{Z}^n$ such that
$$
H({\bf p}_1) \leq H({\bf p}_2) \leq \cdots \leq H({\bf p}_N)
$$
and
$$  
\rank({\bf p}_1,{\bf p}_2, \ldots , {\bf p}_N) < n.
$$
Then, there exist integers $1 \leq j_1 < j_2 < \cdots < j_l$ with
$l \geq N/2^n$ and such that the points
${\bf p}_{j_1},{\bf p}_{j_2}, \ldots , {\bf p}_{j_l}$ 
belong to a hyperplane $\mathcal{H}$ of $\mathbb{Q}^n$ whose height satisfies
$$
H(\mathcal{H}) \leq n! H({\bf p}_{j_1})^n.
$$
\end{lem}

This is \cite[Lemma 8.4]{AdBuMes}.

We are now in a position to establish Theorem \ref{thmNotUeffective}. 

\begin{proof}[Proof of Theorem \ref{thmNotUeffective}] We give here, the proof for the case $a=1$; see the remark following this proof for the explanation of the general case.

Let $F(x)\in\mathbb{Z}[[x]]$ be a $k$-regular power series and $b\geq 2$ be a positive integer. 
Let $H(x)$ be the series given 

In the proof of Theorem \ref{regularkkernel}, we have 
associated to $F(x)$ power series
$G_1(x), \ldots , G_{S} (x)$ and  $H(x)$ such that $H(x)$ is a $\mathbb{Q}$-linear   
combination of $G_1(x), \ldots , G_{S} (x)$ chosen to have the property that $H(1/b)=F(1/b)$.  
Furthermore, there is a matrix $\mathbf{C}_{2,2}(x)$  
described in the proof of Theorem \ref{regularkkernel}  that satisfies  
$$
\left[\begin{matrix} 1\\G_1(x)\\ \vdots\\ G_{S-1}(x)\\ H(x)\end{matrix}\right]
=\left[\begin{matrix} 1 & \mathbf{0}_{1\times S}\\ \mathbf{0}_{S\times 1} 
& \mathbf{C}_{2,2}(x)\end{matrix}\right]
\left[\begin{matrix} 1\\G_1(x^k)\\ \vdots\\ G_{S-1}(x^k)\\ H(x^k)\end{matrix}\right]. 
$$
We pick a positive integer $B$ such that
$B\mathbf{C}_{2,2}(x)$ has entries in $\mathbb{Z}[x]$.  
For clarity, set $d:=S+1$. Note that $d\leq L+1.$

By Lemma \ref{right}, there exist polynomials $P_{1,n}(x),\ldots,P_{d,n}(x),Q_n(x)\in\mathbb{Z}[x]$ 
such that 
\begin{equation}\label{QbQ} 
Q_n(x)=B^nQ_0(x^{k^n}),
\end{equation} 
and positive constants $C_0,C_1$ such that, 
for $i\in\{1,\ldots, d\}$ and for sufficiently large $n$, 
we have $Q_n(1/b)\neq 0$ and 
\begin{equation}\label{y9.1} 
\left|F(1/b)-\frac{P_{d,n}(1/b)}{Q_n(1/b)}\right|
=\left|H(1/b)-\frac{P_{d,n}(1/b)}{Q_n(1/b)}\right|\leq \frac{C_1C_0^n}{b^{d(d+2)Hk^n}}.
\end{equation} 
Write 
$$
Q_0(x):=\sum_{i=0}^D a_i x^{i} = \sum_{i=1}^{\delta} a_{\ell_i} x^{\ell_i},
$$ 
where
$$
\{ i : 0 \le i \le D, a_i \not= 0\} := \{\ell_1, \ldots , \ell_{\delta} \},
$$
with $0 = \ell_1 < \ell_2 < \ldots < \ell_{\delta} = D$. 

Lemma \ref{right} asserts that
\begin{equation}\label{y9.2}
D \le (d^2 + d - 1) H,  
\end{equation}
 where $H$ is the maximum of the degrees of the entries of ${\bf C}_{2,2}(x)$, and that
\begin{equation}\label{y9.2a}
Q_n (x) = B^n Q_0 (x^{k^n}) = \sum_{i=1}^{\delta} a_{\ell_i} B^n x^{\ell_i k^n}.
\end{equation}
Set  
$$
p_n:=b^{Dk^n}P_{d,n}(1/b)\quad\hbox{and}\quad q_n:=b^{Dk^n}Q_n(1/b). 
$$ 
Note that $p_n$ and $q_n$ are both integers. 

It follows from \eqref{QbQ} and \eqref{y9.2a} that
\begin{equation}\label{y9.3}
|q_n|=B^n\left|\sum_{i=1}^{\delta} 
a_{\ell_i}b^{(D-\ell_i)k^n}\right|\leq B^n \sum_{i=1}^{\delta} |a_{\ell_i}| b^{(D-\ell_i)k^n}
\leq C_2 B^n b^{Dk^n}, 
\end{equation}
where $C_2$ is a positive constant depending only on $F$. 
Thus, by \eqref{y9.1}, \eqref{y9.2} and using that $d \ge 2$, we obtain
\begin{equation}\label{L1inf}
\left|q_n F(1/b)-p_n\right|\leq \frac{C_1C_2(C_0 B)^n b^{(d^2 + d - 1)Hk^n}}{b^{d(d+2)Hk^n}}
=\frac{C_1C_2(C_0 B)^n}{b^{dHk^n}}<\frac{1}{b^{Hk^n}}.
\end{equation}

We proceed in two steps. First, we prove that $F(1/b)$ is rational or transcendental.
Second, we show that $F(1/b)$ cannot be too well approximated 
by irrational algebraic numbers of any fixed degree.

For simplicity, we set $\xi = F(1/b)$.

Assume that $\xi$ is algebraic. 
We consider the linear forms in the $\delta + 1$ variables $X_1, \ldots , X_{\delta + 1}$:
$$
L_{i, \infty}  (X_1, \ldots , X_{\delta + 1}) = X_i,  \quad (1 \le i \le \delta ), 
$$
$$
L_{\delta + 1, \infty}  (X_1, \ldots , X_{\delta + 1} ) 
= \sum_{i=1}^{\delta} a_{\ell_i} \xi  X_i - X_{\delta + 1}, 
$$
and, for every prime divisor $p$ of $b$,
$$
L_{i, p} (X_1, \ldots , X_{\delta+ 1}) = X_i,  \quad (1 \le i \le \delta + 1).
$$ 
For $n \ge 1$, set
$$
{\bf x}_n = (B^{n} b^{(D - \ell_1) k^{n}}, \ldots , B^{n} b^{(D - \ell_{\delta}) k^{n}}, p_{n}).
$$

Define $\varepsilon'$ by the equality
$$
\varepsilon' = \min \left\{ {\log 2 \over 2 D \log b}, {1 \over 6 \delta (1 + \omega (b)) d^2} \right\},
$$
where $\omega(b)$ denotes the number of distinct prime factors of $b$. 

Set
\begin{equation}\label{y9.6a}  
e_{1, \infty} = 1,  \quad e_{i , \infty} = {D - \ell_i \over D} + \varepsilon' \quad (2 \le i \le \delta),
\quad \hbox{and} \quad
e_{\delta + 1, \infty} = - {1 \over 4 d^2},
\end{equation}    
and, for every prime divisor $p$ of $b$,
\begin{equation}\label{y9.6b}   
e_{i, p} = {\log |b|_p \over \log b} \cdot {D - \ell_i \over D} + \varepsilon' \quad ( 1 \le i \le \delta),
\quad \hbox{and} \quad
e_{\delta + 1, p} = 0.
\end{equation}   
Recall that the $p$-adic absolute value is normalised such that $|b|_p = p^{-v}$
if $p^v$ divides $b$ and $p^{v+1}$ does not. 
Set $\mathcal{S} = \{\infty\} \cup \{p : \hbox{$p$ divides $b$}\}$.
We check that
$$
e_{i, \infty} \le 1, \quad e_{i, p} \le 0, \quad (1 \le i \le \delta + 1, p \in \mathcal{S} \setminus \{\infty\}),
$$
and 
$$
-\varepsilon := \sum_{p\in S} \sum_{i=1}^{\delta + 1} e_{i, p} \le 
\delta (1 + \omega(b)) \varepsilon' + e_{\delta + 1, \infty} \le
{1 \over 6 d^2} - {1 \over 4 d^2} = - {1 \over 12 d^2}.
$$

Observe that there exists a positive real number $C_3$ such that
\begin{equation}\label{y9.11a}
B^{n} b^{D k^{n}} \le 
\Vert {\bf x}_n \Vert
\le C_3 B^{n} b^{D k^{n}},  
\end{equation}
for $n$ sufficiently large.

We claim that, if $n$ is large enough, then \eqref{L1inf} implies that the tuple
${\bf x}_n$ satisfies the system of inequalities
\begin{equation}\label{y9.12a}
|L_{i, p}({\bf x})|_p\leq \Vert {\bf x} \Vert^{e_{i, p}}\quad (p\in \mathcal{S},\, i=1, \ldots ,  \delta + 1),
\end{equation}
where $| \cdot |_{\infty}$ means the ordinary absolute value.
To see this, first observe that, for any positive real number $\eta$ at most equal to
$(\log 2)/ (2 D \log b)$, we have
\begin{align*}
|L_{i, \infty}({\bf x}_n)| = B^n b^{(D- \ell_i) k^n} & \le (B^n b^{D k^n})^{\eta + (D - \ell_i)/D} \\
& \le \Vert {\bf x}_n \Vert^{\eta + (D - \ell_i)/D},
\end{align*}
for $i = 2, \ldots , \delta$, and
\begin{align*}
|L_{i, p}({\bf x})|_p \le |b|_p^{(D- \ell_i) k^n} 
& \le (C_3 B^n b^{D k^n})^{\eta + (D - \ell_i) \log |b|_p / ( D \log b)} \\
& \le \Vert {\bf x}_n \Vert^{\eta + (D - \ell_i) \log |b|_p / ( D \log b)}
\end{align*}
for $i = 1, \ldots , \delta - 1$ and every prime divisor $p$ of $b$, provided that $n$
is sufficiently large. Since 
$$
|L_{1, \infty}  ({\bf x}_n)| = B^n b^{D k^n} \le \Vert {\bf x}_n \Vert
$$
and
$$
|L_{\delta, p}({\bf x})|_p \le 1, \quad |L_{\delta + 1, p}({\bf x})|_p \le 1,
$$
for every prime divisor $p$ of $b$, it only remains for us to check that
\begin{equation}\label{y9.12aa}
|L_{\delta + 1, \infty}({\bf x}_n)|  \le \Vert {\bf x}_n \Vert^{e_{\delta + 1, \infty}}
\end{equation}
holds for every sufficiently large integer $n$. To this end, observe that the combination of
\eqref{L1inf}, \eqref{y9.2} and $d \ge 2$ gives
$$
|L_{\delta + 1, \infty}({\bf x}_n)| = |q_n F(1/b) - p_n| < b^{-H k^n} \le (b^{D k^n})^{-1 / (2 d^2)},
$$
for $n$ large enough, which, by \eqref{y9.11a}, gives \eqref{y9.12aa}.

By Theorem \ref{thm5.1}, the
tuples ${\bf x}_n$, $n \geq 1$, 
satisfying \eqref{y9.12a} are contained in a finite union of 
proper rational subspaces of $\mathbb{Q}^{\delta + 1}$.
Thus, we deduce that
there exist an infinite set $\mathcal{N}$ of positive integers
and rational integers $y_1, \ldots , y_{\delta + 1}$, not all zero, such that
\begin{equation}\label{y9.12b}
y_1 B^{n} b^{(D - \ell_1) k^{n}} + \ldots 
+ y_{\delta} B^{n } b^{(D - \ell_{\delta}) k^{n}} + 
y_{\delta + 1} p_{n} = 0, 
\end{equation}
for every $n$ in $\mathcal{N}$. 
Dividing \eqref{y9.12b} by $B^n b^{(D - \ell_1) k^{n}}$ and letting $n$
tend to infinity along $\mathcal{N}$, we deduce that 
$$
y_1 + a_0 \xi y_{\delta + 1} = 0.
$$
If $\xi$ is irrational, we get 
that $y_1 = y_{\delta + 1} = 0$ and, using again \eqref{y9.12b}, we deduce that
$$
y_1 = \ldots = y_{\delta + 1} = 0,
$$
a contradiction. Consequently, the real number $\xi$ is either rational or transcendental.

We will now use the full strength of Theorem \ref{thm5.1}, that is, the upper bound 
for the number of subspaces provided by \eqref{y5.9}, 
to estimate from below the distance between $\xi$
and irrational algebraic numbers of any fixed degree.

Let $m \ge 2$ be an integer and $\alpha$ be a real algebraic number
of degree $m$ and sufficiently large height. 
Let $j$ be the integer determined by the inequalities
\begin{equation}\label{y9.8}
b^{d(d+1) H k^{j-1}} \le  H(\alpha) < b^{d (d+1) H k^j}. 
\end{equation}
In the sequel, the integer $j$ is implicitly assumed to be sufficiently
large. This causes no trouble, since we can safely omit the case
$|\xi - \alpha|$ for $\alpha$ of bounded degree and height.

Let us define the real number $\chi$ by
$$
|\xi - \alpha| = \frac{1}{H(\alpha)^{\chi}}.
$$
We will bound $\chi$ from above in terms of $m$.

It follows from \eqref{y9.1} that
\begin{equation}\label{y9.7}
\left| \xi - \frac{p_n}{q_n} \right| < {1 \over 2 b^{d(d+1) H k^n}},
\end{equation}
when $n$ is large enough. 

Let $M$ be the greatest integer for which
$$
2 b^{d(d+1) H k^{j+M-1}} < H(\alpha)^{\chi}.
$$
For every integer $h = 0, \ldots, M-1$, we have
$2 b^{d(d+1) H k^{j+h}} \le 2 b^{d(d+1) H k^{j+M-1}}$, thus, by \eqref{y9.7}, 
\begin{align}
\nonumber\biggl|\alpha - {p_{j+h} \over q_{j+h}} \biggr| &
\le \biggl|\xi - {p_{j+h} \over q_{j+h} } \biggr|
+ |\xi - \alpha|\\
\nonumber& \le 
{1 \over 2 b^{d(d+1) H k^{j+h}}} + \frac{1}{H(\alpha)^{\chi}}\\
\label{y9.9} &< {1 \over b^{d(d+1) H k^{j+h}}}.
\end{align}
By \eqref{y9.2}, \eqref{y9.3} and \eqref{y9.9}, we then have  
\begin{align*}
|q_{j+h} \alpha - p_{j+h}|  &= 
\left|\sum_{i = 1}^{\delta} a_{\ell_i} B^{j+ h} b^{(D - \ell_i) k^{j+h}} \alpha -  p_{j+h}\right| \\
&<  q_{j+h} b^{-d(d+1) H k^{j+h}} \\
&< b^{- H k^{j+h} / 2}.
\end{align*}
We proceed as in the proof of \cite[Theorem 2.1]{BuEv08}.
We consider the linear forms in the $\delta + 1$ variables $X_1, \ldots , X_{\delta + 1}$:
$$
L'_{i, \infty}  (X_1, \ldots , X_{\delta + 1}) = X_i,  \quad (1 \le i \le \delta ),
$$
$$
L'_{\delta + 1, \infty}  (X_1, \ldots , X_{\delta + 1} ) 
= \sum_{i=1}^{\delta} a_{\ell_i} \alpha X_i - X_{\delta + 1},
$$
and, for every prime divisor $p$ of $b$,
$$
L'_{i, p} (X_1, \ldots , X_{\delta+ 1}) = X_i,  \quad (1 \le i \le \delta + 1).
$$ 
For $h = 0, \ldots , M-1$, set
$$
{\bf x}'_h = (B^{j+h} b^{(D - \ell_1) k^{j+h}}, \ldots , B^{j+h} b^{(D - \ell_{\delta}) k^{j+h}}, p_{j+h}).
$$
By \eqref{y9.11a}, if $j$ is large enough, then the tuple
${\bf x}'_h$ satisfies the system of inequalities
\begin{equation}\label{y9.12}
|L'_{i, p}({\bf x})|_p\leq \Vert {\bf x} \Vert^{e_{i, p}}\quad (p\in \mathcal{S},\, i=1, \ldots ,  \delta + 1),
\end{equation}
with the same $e_{i,p}$ as defined above in \eqref{y9.6a} and \eqref{y9.6b}.   
We omit the detailed proof, which is very similar    
to that of \eqref{y9.12a}.   

In the sequel, the constants $c_1, c_2, \ldots$
are independent of the degree $m$ of $\alpha$.
Theorem \ref{thm5.1} and \eqref{y9.8} imply that the
tuples ${\bf x}'_h$ with $\lfloor M / 2 \rfloor  \le h \le M-1$ 
satisfying \eqref{y9.12} are contained in a finite union of $T$
proper rational subspaces of $\mathbb{Q}^{\delta + 1}$, where
\begin{equation}\label{y9.13}
T \le \lceil c_1  \log (8m) \log \log (8m) \rceil .
\end{equation}

Let $U$ be a positive integer 
and assume that one of these proper rational subspaces
contains $2^{\delta + 1} U$ integer tuples ${\bf x}'_h$ 
all of which satisfy \eqref{y9.12}. 
By Lemma \ref{lem8.4}, there exist a nonzero integer tuple $(z_1, \ldots , z_{\delta + 1})$ and
integers  $\lfloor M / 2 \rfloor  \leq i_1 < \cdots < i_U < M$ such that
\begin{equation}\label{y9.15}
z_1 B^{j + i_u} b^{(D - \ell_1) k^{j+i_u}} + \cdots 
+ z_{\delta} B^{j + i_u} b^{(D - \ell_{\delta}) k^{j+i_u}} + 
z_{\delta + 1} p_{j+i_u} = 0, 
\end{equation} for $1 \le u \le U$, and
\begin{equation}\label{y9.16}
Z := \max\{|z_1|, \ldots , |z_{\delta + 1}|\} \le (\delta + 1)!  
\bigl( C_3 B^{j + i_1} b^{D k^{j+i_1}} \bigr)^{\delta + 1}.  
\end{equation}
Note that
$$
\biggl|{p_{j+i_U} \over  a_0 B^{j + i_U} b^{D k^{j+i_U}} }
- \alpha \biggr| 
 \le \biggl| {p_{j+i_U} \over q_{j + i_U}}  
- {p_{j+i_U} \over  a_0 B^{j + i_U} b^{D k^{j+i_U}} } \biggr| + 
\biggl| {p_{j+i_U} \over q_{j + i_U}} 
- \alpha \biggr| \le {2 \over b^{k^{j + i_U} / 2}}.
$$
It then follows from \eqref{y9.15}, with $u=U$, and \eqref{y9.16} that
\begin{equation}\label{y9.17}
|z_1 + z_{\delta + 1} a_0 \alpha| \le {  Z  \over b^{k^{j + i_U} / 3}}.  
\end{equation}
Since $\alpha$ is irrational of degree $m$, Lemma \ref{lem8.2}, \eqref{y9.8} and \eqref{y9.16} 
imply that
\begin{equation}\label{y9.18}
|z_1 + z_{\delta + 1} a_0 \alpha| \ge 2^{-2 m} |a_0 Z|^{-m+1} 
H(\alpha)^{-1} \ge 2^{-2m} |a_0 Z|^{-m+1} b^{- d (d+1) H k^j} .
\end{equation}
Combining \eqref{y9.17} and \eqref{y9.18} gives
$$
b^{k^{j + i_U} / 3} \le   |4 a_0 Z|^m b^{d (d+1) H k^j}
\le b^{2 D^2 m k^{j + i_1}},
$$
for $j$ sufficiently large. This implies that
$$
U \le c_2 \log m.
$$

Thus, each of the $T$ subspaces introduced
above contains at most $c_2 \log m$ elements in the set 
$\left \{ {\bf x}'_h, \; 0 \leq h \leq M-1 \right \}$. 
The upper bound for $T$ given by \eqref{y9.13} then implies
$$
M < c_3   (\log 8 m)^2 (\log \log 8 m).
$$  

Moreover, \eqref{y9.8} and the definition of $\chi$ give
$$
H(\alpha)^{\chi} \le 2 b^{d(d+1) H k^{j+M}} \le
2 H(\alpha)^{k^{M+1}},  
$$
whence
$$
\chi \le 2 k^{M+1} \le (8 m)^{c_4 (\log 8 m) (\log \log 8 m)}.
$$
Consequently, Lemma \ref{lem8.1} implies that 
$$
w_m (\xi) \leq \max\left\{w_1(\xi), 
(8 m)^{c_5  (\log 8 m) (\log \log 8 m)} \right\}, 
$$
for every integer $m \geq 1$.

Furthermore, we have already established that $\xi$ is not
a Liouville number, that is, that $w_1 (\xi)$ is finite. A bound for its value is given in Theorem \ref{regularkkernel}. It then follows
that $w_m (\xi)$ is finite for every positive integer $m$, showing that $\xi$
cannot be a $U$-number. 

Since we have already established that $\xi$ is either rational or transcendental,
this completes the proof of the theorem. 
\end{proof}

\begin{rem} Inequality \eqref{y9.7} is less precise than 
what we have actually obtained. Indeed, we have
$$
\Bigl| \xi - {p_n \over q_n} \Bigr| < {1 \over 2 b^{d(d+2) (1 - \varepsilon) H k^n}},   
$$
for every $\varepsilon > 0$ and every $n$ sufficiently large.
Keeping this in mind, we can prove that the conclusion of Theorem \ref{thmNotUeffective} holds for the real numbers $F(a/b)$ provided that $\rho= (\log |a|)/(\log b) < 1/(L+2)\leq 1/ (d+1)$. To see this,
we proceed exactly as above, enlarging the set $\mathcal{S}$ in such a way that it
contains all the prime divisors of $a$. We need to check that the product of the 
linear forms is sufficiently small, that is, that
$$
(a/b)^{d(d+2 - \varepsilon) H k^n} b^{(d^2 + d - 1) H k^n}
$$
is smaller than $b^{-\eta H k^n}$ for some positive real
number $\eta$. This is true, since
$$
d^2 + d - 1 - d (d+2 - \varepsilon) (1- \rho) < 0,
$$
when $\varepsilon$ is small enough.
\end{rem}

{\bf Acknowledgement.}
We thank the anonymous referee for many helpful comments and suggestions.


\bibliographystyle{amsalpha}

\end{document}